\documentclass[12pt]{amsart}
\usepackage{a4wide,enumerate,xcolor}
\usepackage{amsmath,graphicx}
\allowdisplaybreaks

\usepackage[colorlinks=true,linkcolor=blue,
anchorcolor=black,citecolor=cyan,filecolor=black,menucolor=black,
runcolor=black,urlcolor=blue]{hyperref}

\let\pa\partial
\let\na\nabla
\let\eps\varepsilon
\newcommand{\N}{{\mathbb N}}
\newcommand{\R}{{\mathbb R}}
\newcommand{\diver}{\operatorname{div}}

\newcommand{\dom}{\mathcal{D}}

\newtheorem{theorem}{Theorem}
\newtheorem{lemma}[theorem]{Lemma}

\newtheorem{remark}[theorem]{Remark}

\begin{document}

\title[Poisson--Nernst--Planck--Fermi system for charge transport]{
Analysis of a Poisson--Nernst--Planck--Fermi system \\
for charge transport in ion channels}

\author[A. J\"ungel]{Ansgar J\"ungel}
\address{Institute of Analysis and Scientific Computing, Technische Universit\"at Wien,
Wiedner Hauptstra\ss e 8--10, 1040 Wien, Austria}
\email{juengel@tuwien.ac.at}

\author[A. Massimini]{Annamaria Massimini}
\address{Institute of Analysis and Scientific Computing, Technische Universit\"at Wien,
Wiedner Hauptstra\ss e 8--10, 1040 Wien, Austria}
\email{annamaria.massimini@tuwien.ac.at}

\date{\today}

\thanks{The authors acknowledge partial support from
the Austrian Science Fund (FWF), grants P33010 and F65.
This work has received funding from the European
Research Council (ERC) under the European Union's Horizon 2020 research and innovation programme, ERC Advanced Grant no.~101018153.}

\begin{abstract}
A modified Poisson--Nernst--Planck system in a bounded domain with mixed Dirichlet--Neumann boundary conditions is analyzed. It describes the concentrations of ions immersed in a polar solvent and the correlated electric potential due to the ion--solvent interaction.
The concentrations solve cross-diffusion equations, which are thermodynamically consistent. The considered mixture is saturated, meaning that the sum of the ion and solvent concentrations is constant. The correlated electric potential depends nonlocally on the electric potential and solves the fourth-order Poisson--Fermi equation. The existence of global bounded weak solutions is proved by using the boundedness-by-entropy method. The novelty of the paper is the proof of the weak--strong uniqueness property. In contrast to the existence proof, we include the solvent concentration in the cross-diffusion system, leading to a diffusion matrix with nontrivial kernel. Then the proof is based on the relative entropy method for the extended cross-diffusion system and the positive definiteness of a related diffusion matrix on a subspace.
\end{abstract}



\maketitle


\section{Introduction}

The modeling of the transport of ions through biological channels is of fundamental importance in cell biology. 
Several strategies have been developed in past decades, using
molecular or Brownian dynamics or the
Poisson--Nernst--Planck theory. This theory relies on the assumptions
that the dynamics of ion transport is based on diffusion and electrostatic interaction only and that the solution is dilute. However, the presence of narrow channel pores requires a more sophisticated modeling. In particular, the ion size is not small compared to the biological channel diameter, and many-particle interactions due to the confined geometry need to be taken into account. In this paper, we analyze a modified Poisson--Nernst--Planck system modeling ion--water interactions and finite ion size constraints. We prove the existence of global weak solutions and, as the main novelty, the weak--strong uniqueness property using entropy methods.

\subsection{The model setting}

The evolution of $n$ ionic species, immersed in a solvent (like water), is assumed to be given by the equations 
\begin{align}
  & \pa_t u_i + \diver J_i = r_i(u), \quad 
  J_i = -D_i(\na u_i - u_i\na\log u_0 
  + u_iz_i\na\Phi), \label{1.mass} \\
  & \lambda^2(\ell^2\Delta - 1)\Delta\Phi = \sum_{j=1}^n z_ju_j + f(x)\quad\mbox{in }\Omega,\ t>0,\ i=1,\ldots,n, \label{1.pf}
\end{align}
where $\Omega\subset\R^d$ ($d\ge 1$) is a bounded Lipschitz domain, 
$u=(u_1,\ldots,u_n)$ is the concentration vector, 
supplemented with initial and mixed Dirichlet--Neumann boundary conditions,
\begin{align}
  & u_i(\cdot,0) = u_i^0 \quad\mbox{in }\Omega, \quad i=1,\ldots,n, \label{1.ic} \\
	& J_i\cdot\nu = 0\mbox{ on }\Gamma_N, \quad u_i=u_i^D\mbox{ on }\Gamma_D,\ t>0,
	\label{1.bcu} \\
	& \na\Phi\cdot\nu = \na\Delta\Phi\cdot\nu = 0\mbox{ on }\Gamma_N, \quad
	\Phi = \Phi^D, \ \Delta\Phi = 0 \mbox{ on }\Gamma_D,\ t>0, \label{1.bcphi}
\end{align}
where $\pa\Omega=\Gamma_D\cup\Gamma_N$, $\Gamma_D\cap\Gamma_N=\emptyset$, and
$\nu$ is the exterior unit normal vector to $\pa\Omega$.

The unknowns are the ion concentrations (or volume fractions) $u_i(x,t)$ of the $i$th ion species and the correlated electric potential $\Phi(x,t)$. The solvent concentration (or volume fraction) $u_0(x,t)$ is given by $u_0=1-\sum_{i=1}^n u_i$, which means that the 
mixture is saturated. Equations \eqref{1.mass} are cross-diffusion equations with the fluxes $J_i$ and the reaction rates $r_i(u)$. The parameters are the diffusivities $D_i>0$ and the valences $z_i\in\mathbb{Z}$. 
Equation \eqref{1.pf} is the Poisson--Fermi equation with the scaled Debye length $\lambda>0$, the correlation length $\ell>0$, and the given background charge density $f(x)$. We assume that the domain is isolated on the Neumann boundary, while the concentrations and the electric potential are prescribed on the Dirichlet boundary. 

In the following, we discuss definition \eqref{1.mass} of the fluxes and equation \eqref{1.pf} for the correlated electric potential.

We recover the classical Poisson--Nernst--Planck equations if $u_0=\mbox{const.}$ and $\ell=0$. In this situation, we can write $J_i^{\rm id}=-D_iu_i\na\mu_i^{\rm id}$ with the electrochemical potential $\mu_i^{\rm id}=\log u_i+z_i\Phi$ of an ideal dilute solution. In concentrated solutions, the finite size of the ions needs to be taken into account, expressed by the excess chemical potential $\mu_i^{\rm ex}$, so that the electrochemical potential becomes $\mu_i=\mu_i^{\rm id}+\mu_i^{\rm ex}$. Bikerman \cite{Bik42} suggested the choice $\mu_i^{\rm ex} = -\log(1-\sum_{i=1}^n u_i) = -\log u_0$; also see \cite[Sec.~3.1.2]{BMSA09}. Then $J_i=-D_iu_i\na\mu_i$ coincides with the flux adopted in our model \eqref{1.mass}. Note that solving $\mu_i= \log(u_i/u_0)+z_i\Phi$ for the concentrations, we find that the ion profiles obey the Fermi--Dirac statistics
$$
  u_i = \frac{\exp(\mu_i-z_i\Phi)}{1+\sum_{j=1}^n\exp(\mu_j-z_j\Phi)},
  \quad i=1,\ldots,n.
$$
Then, given $\mu_i$ and $\Phi$, the bounds $0\le u_i\le 1$ are automatically satisfied. Other choices of the excess chemical potential were suggested in \cite[Sec.~2.1]{BiSo07}. 

In the literature, there exist also other approaches to define the fluxes $J_i$ under finite size constraints. The diffusion limit of an on-lattice model, which takes into account that neighboring sites may be occupied (modeling size exclusion), was performed in \cite{BSW12}, analyzed in \cite{GeJu18}, and numerically solved in \cite{CHM23}, resulting to
\begin{equation}\label{1.gersten}
  J^{(1)}_i=-D_i(u_0\na u_i - u_i\na u_0 + u_0u_iz_i\na\Phi), 
  \quad i=1,\ldots,n.
\end{equation}
This model avoids the singular term $\na\log u_0$, which is delicate near $u_0=0$, but it introduces the diffusion term $u_0\na u_i$, which degenerates at $u_0=0$. Another flux definition was suggested in \cite{GoKo18},
$$
  J^{(2)}_i = -D_i\bigg(\na u_i
  + u_iz_i\na\Phi - \sum_{j=1}^n z_ju_j\na\Phi\bigg),
  \quad i=1,\ldots,n.
$$
The additional term $-\sum_{j=1}^nz_ju_j\na\Phi$ comes from the force balance in the Euler momentum equation for zero fluid velocity. 
The ion--water interaction is described in \cite{Che16} by 
$$
  J^{(3)}_i = -D_i\bigg(\na u_i + u_iz_i\na\Phi
  - \frac{\pa\eps_0}{\pa u_i}|\na\Phi|^2\bigg), \quad i=1,\ldots,n,
$$
where the dielectricity $\eps_0=\lambda^2$, instead of being constant, depends on $u$. This assumption is based on the experimental observation that the dielectric response of water decreases as ion concentrations increase \cite{Che16}. Thus, $\pa\eps_0/\pa u_i<0$, showing that the ion--water interaction energy is always nonnegative. Finite ion size effects are modeled in \cite{LiEi14} by including an approximation of the Lennard--Jones potential in the energy functional, leading to 
$$
  J_i^{(4)} = -D_i\bigg(\na u_i + u_i\na\sum_{j=1}^n a_{ij}u_i
  + z_i\na\Phi\bigg), \quad i=1,\ldots,n.
$$
Assuming that $(a_{ij})$ is positive definite, the global existence of weak solutions for two species was proved in \cite{Hsi19}. For the analysis of the stationary equations, see \cite{Gav18}.  
Finally, excluded volume effects can be included by considering nonlinear diffusivities $D_i(u_i)=1+\alpha u_i$, where $\alpha>0$ is a measure of the volume exclusion interactions \cite{BrCh12}.

Our model has the advantage of being consistent with the thermodynamical model \cite{DGM13}
$$
  J_i = -\sum_{j=1}^n D_{ij}u_j\na(\mu_j-\mu_0), \quad\mbox{where }
  \mu_i = \log u_i + z_i\Phi, \quad \mu_0 = \log u_0 + z_0\Phi,
$$
assuming that the diffusion matrix is diagonal, $D_{ij}=D_i\delta_{ij}$, and that the solvent is neutral, $z_0=0$. 

The interaction of the ions with polar solvents like water is modeled by the potential in \eqref{1.pf}. Indeed, let $\phi$ be the electric potential of free ions, given by $-\lambda^2\Delta\phi = \rho$, where $\rho$ is the total charge density. Then the correlated potential $\Phi=\ell^{-2}Y_\ell*\phi$ is the convolution between the Yukawa potential $Y_\ell(x)=(|x|/\ell)^{-1}\exp(-|x|/\ell)$ \cite{LiEi20} and the electric potential, where $\ell>0$ is the correlation length of the screening by ions and water \cite{XLE16}. As this potential satisfies $-\ell^2\Delta\Phi+\Phi=\phi$, we recover \eqref{1.pf} with $\rho=\sum_{j=1}^nz_ju_j+f(x)$. Thus, the Poisson--Fermi equation \eqref{1.pf} includes finite ion size effects and polarization
correlations among water molecules. It generalizes the fourth-order differential permittivity operator of \cite{San06} and the nonlocal permittivity in ionic liquids of \cite{BSK11}.
If there are no correlation and polarization effects ($\ell=0$), we recover the standard Poisson equation for the electric potential. The expression $\eps_0=\lambda^2(\ell^2\Delta-1)$ can be interpreted as a dielectric differential operator. 

\subsection{Entropy structure}

System \eqref{1.mass} can be written as a cross-diffusion system with a diffusion matrix which is neither symmetric nor positive definite. This issue is overcome by exploiting the entropy (or free energy) structure and using the boundedness-by-entropy method \cite{Jue15}. The free energy associated to \eqref{1.mass}--\eqref{1.pf} is given by 
\cite{BSK11} 
\begin{align}
  & H(u) = \int_\Omega h(u)dx, \quad\mbox{where} \label{1.H} \\
	& h(u) = \sum_{i=0}^n\int_{u_i^D}^{u_i}\log\frac{s}{u_i^D}ds
	+ \frac{\lambda^2}{2}|\na(\Phi-\Phi^D)|^2 + \frac{(\lambda\ell)^2}{2}|\Delta(\Phi-\Phi^D)|^2.
	\nonumber
\end{align}
The energy density $h(u)$ consists of the internal, free-ion 
electric, and correlation electric energies. 
The free energy allows us to formulate equations \eqref{1.mass} as a diffusion system with a positive semidefinite diffusion matrix. Indeed, we introduce the electrochemical potentials
$$
  \widetilde{\mu}_i = \frac{\pa h}{\pa u_i} 
  = \log\frac{u_i}{u_0} - \log\frac{u_i^D}{u_0^D} + z_i(\Phi-\Phi^D),
  \quad i=1,\ldots,n,
$$
where $\pa h/\pa u_i$ denotes the variational derivative of $h$ with respect to $u_i$ (see \cite[Lemma 7]{GeJu18}) and
$u_0^D=1-\sum_{i=1}^n u_i^D$. As in \cite{GeJu18}, we split the electrochemical potentials into the entropy variables $w_i$ and the boundary contributions $w_i^D$ by
\begin{equation}\label{1.wi}
  w_i := \log\frac{u_i}{u_0} + z_i\Phi, \quad w_i^D := \log\frac{u_i^D}{u_0^D} + z_i\Phi^D.
\end{equation}
Then equations \eqref{1.mass} can be written as 
\begin{equation}\label{1.B}
  \pa_t u_i - \diver\sum_{j=1}^n B_{ij}(w,\Phi)\na w_j = r_i(u), \quad i=1,\ldots,n,
\end{equation}
where $B_{ij}=D_iu_i\delta_{ij}$ and $u_i=u_i(w,\Phi)$ is interpreted as a function of $w=(w_1,\ldots,w_n)$ and $\Phi$ according to
\begin{equation}\label{1.ui}
  u_i(w,\Phi) = \frac{\exp(w_i-z_i\Phi)}{1+\sum_{j=1}^n\exp(w_i-z_j\Phi)}.
\end{equation}
The advantage of formulation \eqref{1.B} is that the new diffusion matrix $B=(B_{ij})$ is symmetric and positive semidefinite. Observe that system \eqref{1.B} is of degenerate type since $u_i=0$ is possible, and $\det B=0$ in this case.
The formulation in terms of entropy variables has the further advantage that the ion concentrations $u_i$, defined by \eqref{1.ui}, are nonnegative and satisfy $\sum_{i=1}^n u_i\le 1$, thus fulfilling the saturation assumption. 

\subsection{Main results}

We introduce the simplex 
$\dom=\{u=(u_1,\ldots,u_n)\in(0,1)^n:\sum_{i=1}^n u_i<1\}$
and set $\Omega_T=\Omega\times(0,T)$.
The following hypotheses are imposed:
\begin{itemize}
\item[(H1)] Domain: $\Omega\subset\R^d$ ($1\le d\le 3$) is a bounded Lipschitz domain with $\pa\Omega=\Gamma_D\cup\Gamma_N$, $\Gamma_D\cap\Gamma_N=\emptyset$, $\Gamma_N$ is
open in $\pa\Omega$, and $\operatorname{meas}(\Gamma_D)>0$.
\item[(H2)] Data: $T>0$, $D_i>0$, $z_i\in\R$ for $i=1,\ldots,n$, $f\in L^2(\Omega)$.
\item[(H3)] Initial data: $u^0=(u_1^0,\ldots,u_n^0)\in L^1(\Omega;\R^n)$ satisfies $u^0(x)\in\overline\dom$ for a.e.\ $x\in\Omega$.
\item[(H4)] Boundary data: $u^D=(u_1^D,\ldots,u_n^D)\in H^1(\Omega;\R^n)$ satisfies
$u^D(x)\in\dom$ for $x\in\Omega$, 
$\log u_0^D\in L^2(\Omega)$,
and $\Phi^D\in H^2(\Omega)$ solves
\begin{align}\label{1.PhiD}
  & \lambda^2(\ell^2\Delta-1)\Delta\Phi^D = f(x)\mbox{ in }\Omega, \\
	& \na\Phi^D\cdot\nu=\na\Delta\Phi^D\cdot\nu=0\mbox{ on }\Gamma_N, \quad
	\Delta\Phi^D=0\mbox{ on }\Gamma_D. \nonumber
\end{align}
\item[(H5)] Reaction rates: $r_i\in C^0([0,\infty)^n;\R)$
for $i=1,\ldots,n$, and there exists $C_r>0$ such that
for all $u\in L^\infty(\Omega_T;\dom)$ and $\Phi$, given by \eqref{1.pf}
and \eqref{1.bcphi},
\begin{equation}\label{1.H5}
  \int_\Omega\sum_{i=1}^n r_i(u)\frac{\pa h}{\pa u_i} dx
  \le C_r(1+H(u)).
\end{equation}
\end{itemize}

The restriction to three space dimensions in Hypothesis (H1) is not needed. It can be removed by regularizing the Poisson--Fermi equation \eqref{1.pf} to ensure that $\Phi\in L^\infty(\Omega)$; see Remark \ref{rem.Phi} below.
In Hypothesis (H4), it is sufficient to define the boundary data on $\Gamma_D$. We have extended them to $\Omega$ with the special extention of $\Phi^D$, fulfilling the fourth-order elliptic problem \eqref{1.PhiD}. This extension is needed to be consistent with the definition of the free energy and the entropy variables; see \cite[Lemma 7]{GeJu18}. The bound in Hypothesis (H5) is needed to derive gradient bounds on the concentrations from the free energy inequality; see \eqref{1.aux} below. Since $\pa h/\pa u_i$ contains the logarithm, $r_i(u)$ needs to cancel the singularity in $\pa h/\pa u_i$ at $u_i=0$. The contribution of the potential $\Phi-\Phi^D$ in  $\pa h/\pa u_i$ is treated by applying the Poincar\'e inequality and observing that $\int_\Omega|\na(\Phi-\Phi^D)|^2dx\le CH(u)$ for some $C>0$. Therefore, we need the integrated version \eqref{1.H5} instead of the pointwise inequality assumed in \cite[Sec.~1.4]{Jue15}.

We introduce the test spaces
\begin{align*}
  H^1_{D}(\Omega) &= \{v\in H^1(\Omega):v=0\mbox{ on }\Gamma_D\}, \\
  H^2_{D,N}(\Omega) &= \{v\in H^2(\Omega):v=0\mbox{ on }\Gamma_D,\ 
  \na v\cdot\nu=0\mbox{ on }\Gamma_N\}.
\end{align*}
Our first main result is as follows.

\begin{theorem}[Global existence of solutions]\label{thm.ex}
Let Hypotheses (H1)--(H5) hold. Then there exists a bounded weak solution
$u_1,\ldots,u_n$ to \eqref{1.mass}--\eqref{1.bcphi} satisfying
$u_i(x,t)\in\overline{\dom}$ for a.e.\ $(x,t)\in\Omega_T$, $i=1,\ldots,n$,
\begin{align*}
  & \sqrt{u_i}\in L^2(0,T;H^1(\Omega)), \quad u_i\in H^1(0,T;H^1_D(\Omega)')\cap C^0([0,T];L^2(\Omega)), \\
  & \Phi\in L^2(0,T;H^2(\Omega)), \quad\log u_0\in L^2(0,T;H^1(\Omega)),
\end{align*}
the weak formulation
\begin{align}
  \int_0^T\langle\pa_t u_i,\phi_i\rangle dt - \int_0^T\int_\Omega J_i\cdot\na\phi_i dxdt
	&= \int_0^T\int_\Omega r_i(u)\phi_i dxdt, \label{1.weak1} \\
	\lambda^2\int_0^T\int_\Omega(\ell^2\Delta\Phi\Delta\theta + \na\Phi\cdot\na\theta)dxdt
	&= \int_0^T\int_\Omega\bigg(\sum_{i=1}^nz_iu_i + f\bigg)\theta dxdt
    \label{1.weak2}
\end{align}
for all $\phi_i\in L^2(0,T;H^1_D(\Omega))$ and $\theta\in L^2(0,T;H^2_{D,N}(\Omega))$, where $J_i$ is given by \eqref{1.mass} and $\langle\cdot,\cdot\rangle$ is the dual product between $H_D^1(\Omega)'$ and $H_D^1(\Omega)$.
The initial conditions \eqref{1.ic} are satisfied a.e.\ in $\Omega$, and the Dirichlet boundary conditions are fulfilled in the sense of traces in $L^2(\Gamma_D)$. Furthermore, if $r_i(u)=0$ for all $i=1,\ldots,n$ and the Dirichlet boundary conditions are in thermal equilibrium (e.g.\ $w_i^D:=\log(u_i^D/u_0^D)+z_i\Phi^D=\mbox{const.}$ in $\Omega$), 
the solution satisfies for $0<s<t<T$ the free energy inequality
\begin{equation}\label{1.fei}
  H(u(t)) + \int_s^t\int_\Omega\sum_{i=1}^n D_i u_i\bigg|
  \na\bigg(\log\frac{u_i}{u_0}+z_i\Phi\bigg)\bigg|^2 dxd\sigma
  \le H(u(s)).
\end{equation}
The energy dissipation is understood in the sense
$$
  u_i\bigg|\na\bigg(\log\frac{u_i}{u_0}+z_i\Phi\bigg)\bigg|^2
  = \big|2\na\sqrt{u_i} - \sqrt{u_i}\na\log u_0 + \sqrt{u_i}z_i\na\Phi
  \big|^2.
$$
\end{theorem}

The assumption of thermal equilibrium at the Dirichlet boundary, also required in \cite{GeJu18}, is needed to avoid expressions involving $\na w_i^D$ in the free energy inequality. Thus, this condition, together with vanishing reactions, is natural to obtain the monotonicity of the free energy. We refer to Remark \ref{rem.feir} for the case with reaction rates. 

The proof of Theorem \ref{thm.ex} is, similarly as in \cite{GeJu18}, based on an approximation procedure, where we regularize \eqref{1.B} by an implicit Euler approximation and higher-order terms in the entropy variables. The uniform estimates that are needed to perform the de-regularization limit are derived from the free energy inequality, which (without regularization) reads as
\begin{equation}\label{1.aux}
  \frac{dH}{dt} + \int_\Omega\sum_{i=1}^n D_iu_i|\na w_i|^2 dx
  \le \int_\Omega\sum_{i=1}^n r_i(u)\cdot\frac{\pa h}{\pa u_i}dx
  \le C_r(1 + H(u)),
\end{equation}
recalling definition \eqref{1.wi} of $w_i$,
and we can conclude by Gronwall's lemma.
The free energy dissipation term on the left-hand side can be estimated from above by (see Lemma \ref{lem.ineq})
$$
  \int_\Omega u_i|\na w_i|^2 dx 
  \ge \frac12\int_\Omega\big(|\na\sqrt{u_i}|^2
  + |\na\log u_0|^2 + |\na u_0|^2\big)dx - C\int_\Omega|\na\Phi|^2dx.
$$
The last term is bounded by the electric energy part in $H(u)$,
thus giving $H^1(\Omega)$ bounds for $u_i$ for $i=0,\ldots,n$ and 
$\log u_0$. Compared to \cite{GeJu18}, we obtain gradient estimates for all the ion concentrations, but we have to deal with the singular term $\na\log u_0$ in \eqref{1.mass}.
Moreover, compared to \cite{GaFu22}, where a similar Nernst--Planck system (with $\ell=0$) was investigated, we do not need any positivity condition on the initial solvent concentration. 

While the existence proof relies on standard entropy methods, we need a new idea to prove the weak--strong uniqueness result. The uniqueness of {\em weak} solutions is an intricate problem. A uniqueness result for \eqref{1.mass} with the fluxes \eqref{1.gersten} was shown in \cite{GeJu18} for the case $D_i=D$ and $z_i=z$ for all $i$. In this simplified situation, the solvent concentration solves a Poisson--Nernst--Planck system for which the uniqueness of bounded weak solutions can be proved by a combination of $L^2(\Omega)$ estimates and Gajewski's entropy method. This strategy cannot be used for our system. In fact, we need the $H^{-1}(\Omega)$ method and a strong regularity condition for $\na\Phi$, which restricts the geometry of the Dirichlet--Neumann boundary conditions; see Remark \ref{rem.uniq}.
Therefore, we do not aim to prove the uniqueness of weak solutions but the weak--strong uniqueness property only, which has the advantage that we may allow for different coefficients $D_i$ and $z_i$. The weak--strong uniqueness property means that any weak solution to system \eqref{1.mass}--\eqref{1.bcphi} coincides with a strong solution emanating from the same initial conditions as long as the latter exists.
We say that $(\bar{u},\bar\Phi)$ is a {\em strong solution} to
\eqref{1.mass}--\eqref{1.bcphi} if it is a weak solution and 
$$
  \bar{u}_i\ge c>0\mbox{ in }\Omega_T, \quad
  \bar{u}_i,\,\Phi\in L^\infty(0,T;W^{1,\infty}(\Omega))
  \quad\mbox{for all }i=1,\ldots,n.
$$ 
Our second main result is contained in the following theorem.

\begin{theorem}[Weak--strong uniqueness]\label{thm.wsu}
Let the Dirichlet boundary data be in thermal equilibrium in the sense of Theorem \ref{thm.ex} and let $r_i=0$ for $i=1,\ldots,n$.
Let $(u,\Phi)$ be a weak solution and $(\bar{u},\bar\Phi)$
be a strong solution to \eqref{1.mass}--\eqref{1.bcphi}.
Then $u(x,t)=\bar{u}(x,t)$, $\Phi(x,t)=\bar\Phi(x,t)$
for a.e.\ $x\in\Omega$ and $t\in(0,T)$.
\end{theorem}

If the reaction rates are Lipschitz continuous and satisfy some sign conditions, Theorem \ref{thm.wsu} still holds; see Remark \ref{rem.reac}. The condition that the Dirichlet boundary data is in thermal equilibrium is actually not needed, since in contrast to \eqref{1.fei}, the terms involving $\na w_i^D$ cancel out in the computations for the relative free energy
\begin{align*}
  & H(u,\Phi|\bar{u},\bar{\Phi}) = \int_\Omega\big(h_1(u|\bar{u})
  + h_2(\Phi|\bar{\Phi})\big)dx, \quad\mbox{where} \\
  & h_1(u|\bar{u}) = \sum_{i=0}^n\bigg(u_i\log\frac{u_i}{\bar{u}_i}
 - (u_i-\bar{u}_i)\bigg), \\
  & h_2(\Phi|\bar{\Phi}) = \frac{\lambda^2}{2}\big(
  |\na(\Phi-\bar{\Phi})|^2 + \ell^2|\Delta(\Phi-\bar{\Phi})|^2\big),
\end{align*}
which can be identified as the Bregman distance of the free energy.
The key idea of the proof of Theorem \ref{thm.wsu} is to consider the solvent concentration $u_0$ as an independent variable and to formulate the parabolic equations for the extended concentration vector $U=(u_0,u_1,\ldots,u_n)$, leading to 
$$
  \pa_t u_i = \diver\sum_{j=0}^n(A_{ij}(U)\na\log u_j
  + Q_{ij}(U)\na\Phi), \quad i=0,\ldots,n,
$$
where $A_{ij}(U)$ and $Q_{ij}(U)$ depend linearly on $U$ (see \eqref{3.AP}). The matrix $(A_{ij}/\sqrt{u_iu_j}) \in\R^{(n+1)\times(n+1)}$ is positive definite only on the subspace $L=\{y\in\R^{n+1}:
\sum_{i=0}^n\sqrt{u_i}y_i=0\}$. This situation is similar to the Maxwell--Stefan system; see \cite{HJT22}. The time derivative of the relative free energy equals
\begin{align*}
  \frac{dH}{dt}(u,\Phi|\bar{u},\bar\Phi) = K_1 + K_2, 
  \quad\mbox{where }
  K_1 = -\int_\Omega\sum_{j=0}^n A_{ij}\na\log\frac{u_i}{\bar{u}_i}
  \cdot\na\log\frac{u_j}{\bar{u}_j}dx,
\end{align*}
and $K_2$ contains differences like $U_i-\bar{U}_i$ and $\Phi-\bar\Phi$. The properties of the matrices $(A_{ij})$ and $(Q_{ij})$ imply that
\begin{equation*}
  K_1 \le -\min_{i=1,\ldots,n}D_i
  \int_\Omega\bigg(\frac{1}{u_0}|(P_LY)_0|^2
  + \sum_{i=1}^n|(P_LY)_i|^2\bigg)dx,
\end{equation*}
where $P_L$ is the projection on $L$ and $Y_i=\sqrt{u_i}\na\log(u_i/\bar{u}_i)$, as well as
for any $\delta>0$,
\begin{align*}
  K_2 &\le \delta\int_\Omega\bigg(\frac{1}{u_0}|(P_LY)_0|^2
    + \sum_{i=1}^n|(P_LY)_i|^2\bigg)dx \\
  &\phantom{xx}+ C(\delta)\bigg(\sum_{i=0}^n\|u_i-\bar{u}_i\|_{L^2(\Omega)}^2
  + \|\na(\Phi-\bar\Phi)\|_{L^2(\Omega)}^2\bigg).
\end{align*}
Consequently, choosing $\delta>0$ sufficiently small,
\begin{equation*}
  \frac{dH}{dt}(u,\Phi|\bar{u},\bar\Phi)
  \le C\bigg(\sum_{i=0}^n\|u_i-\bar{u}_i\|_{L^2(\Omega)}^2
  + \|\na(\Phi-\bar\Phi)\|_{L^2(\Omega)}^2\bigg)
  \le CH(u,\Phi|\bar{u},\bar\Phi)
\end{equation*}
for some constant $C>0$. Since the initial data of $u$ and $\bar{u}$ coincide, we have $H((u,\Phi)(t)|$ $(\bar{u},\bar\Phi)(t))=0$ and
finally $u(t)=\bar{u}(t)$ and $\Phi(t)=\bar\Phi(t)$ for all $t>0$.
The idea to consider the parabolic system for the extended solution vector $U=(u_0,\ldots,u_n)$ instead of $u=(u_1,\ldots,u_n)$ is the main novelty of this paper.

The article is organized as follows.
The proof of Theorem \ref{thm.ex} is presented in Section \ref{sec.ex}, while Section \ref{sec.wsu} contains the proof of Theorem \ref{thm.wsu}.
We make some remarks  on the uniqueness of solutions in Section \ref{sec.rem}.


\section{Proof of Theorem \ref{thm.ex}}\label{sec.ex}

We assume throughout this section that Hypotheses (H1)--(H5) hold.

\subsection{Solution of an approximate system}
 
We define the approximate problem by the implicit Euler scheme and using a higher-order regularization. Let $T>0$, $N\in\N$, $\tau=T/N$, and $m\in\N$ with $m>d/2$. We assume that $u_i^D\ge \eta>0$ for $i=0,\ldots,n$. Then $w_i^D=\log(u_i^D/u_0^D)+z_i\Phi^D\in H^1(\Omega;\R^n)\cap L^\infty(\Omega;\R^n)$. Since the entropy variables are not needed in the weak formulation \eqref{1.weak1}--\eqref{1.weak2}, we can pass to the limit $\eta\to 0$ at the end of the proof, thus requiring only $u_i^D>0$. Let $k\ge 1$ and let
$u^{k-1}-u^D\in H^1_D(\Omega;\R^n)\cap L^\infty(\Omega;\R^n)$ and $\Phi^{k-1}-\Phi^D\in H_{D,N}^2(\Omega)$ be given. If $k=1$, $\Phi^0\in H^2(\Omega)$ is the unique solution to \eqref{1.pf} with $u_j^0$ instead of $u_j$ on the right-hand side and satisfying the corresponding boundary conditions
in \eqref{1.bcu}--\eqref{1.bcphi}. We wish to find a solution $v^k\in
X:=H^m(\Omega;\R^n)\cap H_D^1(\Omega;\R^n)$ and $\Phi^k-\Phi^D\in H_{D,N}^2(\Omega)$ to
\begin{align}\label{2.approx1}
  \frac{1}{\tau}\int_\Omega & (u^k-u^{k-1})\cdot\phi dx
	+ \int_\Omega \na\phi:B(v^k+w^D,\Phi^k)\na(v^k+w^D)dx \\
	&{}+ \eps\int_\Omega\bigg(\sum_{|\alpha|=m}D^\alpha v^k\cdot D^\alpha\phi 
	+ v^k\cdot\phi\bigg)dx = \int_\Omega r(u^k)\cdot\phi dx, \nonumber \\
	\lambda^2\int_\Omega & \big(\ell^2\Delta\Phi^k\Delta\theta
	+ \na\Phi^k\cdot\na\theta\big) dx
	= \int_\Omega\bigg(\sum_{i=1}^n z_iu_i^k + f\bigg)\theta dx \label{2.approx2}
\end{align}
for all $\phi\in X$ and $\theta\in H_{D,N}^2(\Omega)$.
Here, we have set $u^k:=u(v^k+w^D,\Phi^k)$, where $u(w,\Phi)$ is defined by \eqref{1.ui}, $B_{ij}(w,\Phi)=D_i u_i(w,\Phi)\delta_{ij}$, $r(u)=(r_1(u),\ldots,r_n(u))$, and 
$D^\alpha=\pa^{|\alpha|}/\pa x_1^{\alpha_1}\cdots\pa x_d^{\alpha_d}$ is a partial derivative
of order $|\alpha|=\alpha_1+\cdots+\alpha_d$. Thanks to the higher-order regularization, we obtain approximate solutions $w^k:=v^k+w^D\in H^m(\Omega;\R^n)\hookrightarrow L^\infty(\Omega;\R^n)$. Moreover, since $d\le 3$, we have $\Phi^k\in H^2(\Omega)\hookrightarrow L^\infty(\Omega)$. Hence, $u_i(w^k,\Phi^k)$ is well defined and integrable.

\begin{remark}\label{rem.Phi}\rm
Adding a higher-order regularization to the Poisson--Fermi equation \eqref{2.approx2}, we may obtain $\Phi^k\in L^\infty(\Omega)$ by a Sobolev embedding similarly as for $w^k$. This allows us to remove the restriction $d\le 3$ in Hypothesis (H1). 
\qed\end{remark}

\begin{lemma}
There exists a unique solution $v^k\in H^m(\Omega;\R^n)\cap H_D^1(\Omega;\R^n)$ and $\Phi^k-\Phi^D\in H^2_{D,N}(\Omega)$ to \eqref{2.approx1}--\eqref{2.approx2}.
\end{lemma}

\begin{proof}
The proof is similar to that one of Lemma 5 in \cite{GeJu18}, therefore we give a sketch only.
Let $y\in L^\infty(\Omega;\R^n)$ and $\sigma\in[0,1]$. Let $\Phi^k\in H^2(\Omega)$ be
the unique solution to 
$$
  \lambda^2(\ell^2\Delta-1)\Delta\Phi^k = \sum_{i=1}^n z_iu_i(y+w^D,\Phi^k) + f(x)
	\quad\mbox{in }\Omega
$$
subject to the boundary conditions \eqref{1.bcphi}.
This follows from the fact that the function $(x,\Phi)\mapsto u_i(w(x),\Phi)$ is bounded 
with values in $(0,1)$ and Lipschitz continuous in $\Phi$.
By the Lax--Milgram lemma, there exists a unique solution 
$v\in X$ to the linear problem
\begin{align}\label{2.aux}
  \eps\int_\Omega&\bigg(\sum_{|\alpha|=m}D^\alpha v\cdot D^\alpha\phi + v\cdot\phi\bigg)ds
	+ \int_\Omega\na\phi:B(y+w^D,\Phi^k)\na v dx \\
	&= \delta\int_\Omega r(u(y+w^D,\Phi^k))\cdot\phi dx
	- \delta\int_\Omega\na\phi:B(y+w^D,\Phi^k)\na w^D dx \nonumber \\
	&\phantom{xx}{}- \frac{\delta}{\tau}\int_\Omega\big(u(y+w^D,\Phi^k)-u^{k-1}\big)
	\cdot\phi dx. \nonumber
\end{align}
Indeed, as $B$ is positive semidefinite, the left-hand side is coercive in $H^m(\Omega;\R^n)$.

This defines the fixed-point operator $S:L^\infty(\Omega;\R^n)\times[0,1]\to 
L^\infty(\Omega;\R^n)$, $S(y,\delta)=v$. Then $S(y,0)=0$, $S$ is continuous and, because
of the compact embedding $H^m(\Omega;\R^n)\hookrightarrow L^\infty(\Omega;\R^n)$, also
compact. Using $\phi=v$ as a test function in \eqref{2.aux}, standard estimates lead to
$\eps\|v\|_{H^m(\Omega)}^2 \le C(\tau)\|v\|_{H^m(\Omega)}$,
giving a bound for $v$ in $H^m(\Omega;\R^n)$ uniform in $\delta$. Hence, all fixed
points of $S(\cdot,\delta)$ are uniformly bounded in $L^\infty(\Omega;\R^n)$. We infer
from the Leray--Schauder fixed-point theorem that there exists $v^k\in X$ such that
$S(v^k,1)=v^k$. Then $(v^k,\Phi^k)$ is a solution to \eqref{2.approx1}--\eqref{2.approx2}.
\end{proof}

\subsection{Uniform estimates}

We deduce estimates uniform in $(\eps,\tau)$ from the following free energy inequality.

\begin{lemma}[Discrete free energy inequality]
Let $(v^k,\Phi^k)$ be a solution to \eqref{2.approx1}--\eqref{2.approx2} and set $w^k:=v^k+w^D$ and $u^k:=u(w^k,\Phi^k)$. Then
\begin{align}\label{2.dei}
  H(u^k) - H(u^{k-1}) &+ \frac{\tau}{2}\int_\Omega\sum_{i=1}^n 
  D_i u^k_i|\na w_i^k|^2 dx + \eps\tau\|w^k-w^D\|_{H^m(\Omega)}^2 \\
  &\le \tau C_r(1+H(u^k)) + \frac{\tau}{2}\int_\Omega\sum_{i=1}^n
  D_i |\na w^D_i|^2dx, \nonumber
\end{align}
where $H$ is defined in \eqref{1.H} and $C_r>0$ is introduced in Hypothesis (H5).
\end{lemma}

\begin{proof}
We choose $\phi=\tau v^k=\tau(w^k-w^D)\in X$ as a test function in \eqref{2.approx1}.
Using the generalized Poincar\'e inequality to estimate the $\eps$-regularization
and Hypothesis (H5) to estimate the reaction rates, we find that
\begin{align*}
  \int_\Omega(u^k-u^{k-1})\cdot(w^k-w^D)dx &+ \tau\int_\Omega\na(w^k-w^D):B(w^k,\Phi^k)
	\na w^k dx \\
	&+ \eps\tau C\|w^k-w^D\|_{H^m(\Omega)}^2 \le \tau C_r(1+H(u^k)).
\end{align*}
It follows from the convexity of the function $g(u)=\sum_{i=0}^n\int_{u_i^D}^{u_i}\log(s/u_i^D)ds$ 
and the Poisson--Fermi equation \eqref{1.pf} as in \cite[Section 2]{GeJu18} that
\begin{align*}
  \int_\Omega&(u^k-u^{k-1})\cdot(w^k-w^D)dx
	= \int_\Omega\sum_{i=1}^n(u_i^k-u_i^{k-1})
	\bigg(\log\frac{u_i^k}{u_0^k} - \log\frac{u_i^D}{u_0^D}\bigg)dx \\
	&\phantom{xx}{}+ \int_\Omega\sum_{i=1}^n z_i(u_i^k-u_i^{k-1})(\Phi^k-\Phi^D)dx \\
	&\ge \int_\Omega\big(g(u^k)-g(u^{k-1})\big)dx 
	+ \frac{\lambda^2}{2}\int_\Omega\big(\ell^2|\Delta(\Phi^k-\Phi^D)|^2 
	+ |\na(\Phi^k-\Phi^D)|^2\big)dx \\
	&\phantom{xx}{}- \frac{\lambda^2}{2}\int_\Omega\big(\ell^2|\Delta(\Phi^{k-1}-\Phi^D)|^2 
	+ |\na(\Phi^{k-1}-\Phi^D)|^2\big)dx
	= H(u^k)-H(u^{k-1}).
\end{align*}
Inserting the definition $B_{ij}(w^k,\Phi^k) = D_i u_i^k\delta_{ij}$, we infer from Young's inequality that
\begin{align*}
  \na(w^k-w^D):B(w^k,\Phi^k)\na w^k
	&= \sum_{i=1}^n D_i u_i^k\na(w_i^k-w_i^D)\cdot\na w_i^k \\
	&\ge \frac12\sum_{i=1}^n D_i u_i^k|\na w_i^k|^2 - \frac12 \sum_{i=1}^n D_i u_i^k|\na w_i^D|^2.
\end{align*}
Collecting these estimates and observing that $u_i^k\le 1$ concludes the proof.
\end{proof}

We sum \eqref{2.dei} over $k=1,\ldots,j$,
\begin{align*}
  (1-\tau C_r)H(u^j)
	&+ \frac{\tau}{2}\sum_{k=1}^j\int_\Omega\sum_{i=1}^n D_i u_i^k
	|\na w_i^k|^2 dx
	+ \eps\tau\sum_{k=1}^j\|w^k-w^D\|_{H^m(\Omega)}^2 \\
	&\le \tau C_r\sum_{k=1}^{j-1} H(u^k) + H(u^0) + j\tau C_r 
	+ \frac{\tau}{2}\sum_{k=1}^j\int_\Omega\sum_{i=1}^n D_i|\na w_i^D|^2 dx,
\end{align*}
and, assuming $\tau<1/C_r$, apply the discrete Gronwall inequality \cite{Cla87}:
\begin{align*}
  H(u^j) + \frac{\tau}{2}\bigg(\min_{i=1,\ldots,n}D_i\bigg)\sum_{k=1}^j
  \int_\Omega\sum_{i=1}^n u_i^k|\na w_i^k|^2 dx
  + \eps\tau\sum_{k=1}^j\|w^k-w^D\|_{H^m(\Omega)}^2 \le C(T),
\end{align*}
where $C(T)>0$ does not depend on $(\eps,\tau)$. We still need to
bound the second term on the left-hand side from below.

\begin{lemma}\label{lem.ineq}
It holds that 
\begin{align*}
  \sum_{k=1}^N\tau\int_\Omega\sum_{i=1}^n u_i^k|\na w_i^k|^2 dx 
  &\ge \frac12\sum_{k=1}^N\tau\int_\Omega\bigg(\sum_{i=1}^n
  |\na(u_i^k)^{1/2}|^2 + |\na\log u_0^k|^2 + |\na u_0^k|^2\bigg)dx \\
  &\phantom{xx}- C\sum_{k=1}^N\tau\int_\Omega|\na\Phi^k|^2dx,
\end{align*}
where $C>0$ depends on $(D_i)$ and $(z_i)$.
\end{lemma}

\begin{proof}
We infer from Young's inequality and the bound $u_i^k\le 1$ that
$$
  u_i^k|\na w_i^k|^2 = u_i^k\bigg|\na\log\frac{u_i^k}{u_0^k} + z_i\na\Phi^k\bigg|^2
  \ge \frac12u_i^k\bigg|\na\log\frac{u_i^k}{u_0^k}\bigg|^2
  - |z_i\na\Phi^k|^2.
$$
The first term on the right-hand side is rewritten as
\begin{align*}
  \frac12 u_i^k\bigg|\na\log\frac{u_i^k}{u_0^k}\bigg|^2
  &= \frac12\sum_{i=1}^n\frac{|\na u_i^k|^2}{u_i^k} 
  + \frac12\sum_{i=1}^n u_i^k|\na\log u_0^k|^2 
  - \sum_{i=1}^n\na u_i^k\cdot\na\log u_0^k \\
  &= \frac12\sum_{i=1}^n\frac{|\na u_i^k|^2}{u_i^k} 
  + \frac12(1-u_0^k)|\na\log u_0^k|^2 
  - \na(1-u_0^k)\cdot\na\log u_0^k \\
  &= \frac12\sum_{i=1}^n\frac{|\na u_i^k|^2}{u_i^k}
  + \frac12|\na\log u_0^k|^2 + \frac{|\na u_0^k|^2}{2u_0^k} \\
  &\ge 2\sum_{i=1}^n|\na(u_i^k)^{1/2}|^2 + \frac12|\na\log u_0^k|^2
  + \frac12|\na u_0^k|^2,
\end{align*}
using $u_0^k\le 1$ in the last step.
\end{proof}

Since the free energy is bounded from below, 
we conclude the following uniform bounds.

\begin{lemma}\label{lem.est}
There exists $C>0$ not depending on $(\eps,\tau)$ such that
for $i=1,\ldots,n$,
\begin{align*}
  \sum_{k=1}^N\tau\big(\|(u_i^k)^{1/2}\|_{H^1(\Omega)}^2
  + \|u_i^k\|_{H^1(\Omega)}^2
  + \|u_0^k\|_{H^1(\Omega)}^2 + \|\log u_0^k\|_{H^1(\Omega)}^2\big)
  &\le C, \\
  \eps\sum_{k=1}^N\tau\|w_i^k\|_{H^m(\Omega)}^2
  + \sum_{k=1}^N\tau\|\Phi^k\|_{H^2(\Omega)}^2 &\le C.
\end{align*}
\end{lemma}

\begin{proof}
The inequality 
$$
  \|\na u_i^k\|_{L^2(\Omega)} 
  \le 2\|(u_i^k)^{1/2}\|_{L^\infty(\Omega)}
  \|\na(u_i^k)^{1/2}\|_{L^2(\Omega)}\le 2\|\na(u_i^k)^{1/2}\|_{L^2(\Omega)}
$$ 
shows that $\sum_{k=1}^N\tau\|\na u_i^k\|_{L^2(\Omega)}^2\le C$. 
The $H^2(\Omega)$ bound for $\Phi^k$ follows immediately from the
Poisson--Fermi equation as its right-hand side is bounded in $L^2(\Omega)$.
The $H^1(\Omega)$ bound for $\log u_0^k$ is a consequence of the
$L^2(\Omega)$ bound for $\na\log u_0^k$ and the Poincar\'e inequality,
using the fact that $\log u_0^D\in L^2(\Omega)$ by Hypothesis (H4). 
\end{proof}

\subsection{Limit $(\eps,\tau)\to 0$}

We introduce the piecewise constant in time functions
$u_i^{(\tau)}(x,t) = u_i^k(x)$, $w_i^{(\tau)}(x,t) = w_i^k(x)$, and
$\Phi^{(\tau)}(x,t) = \Phi^k(x)$ for $x\in\Omega$, $t\in((k-1)\tau,k\tau]$. At time $t=0$, we set
$w^{(\tau)}(\cdot,0) = h'(u^0)$ and $u_i^{(\tau)}(\cdot,0) = u_i^0$.
Furthermore, we introduce the shift operator 
$(\sigma_\tau u^{(\tau)})(\cdot,t) = u^{k-1}$ for $t\in((k-1)\tau,k\tau]$. Then, summing \eqref{2.approx1}--\eqref{2.approx2} over $k=1,\ldots,N$, we see that $(u^{(\tau)},\Phi^{(\tau)})$ solves
\begin{align}\label{2.tau1}
  \frac{1}{\tau}&\int_0^T\int_\Omega(u^{(\tau)}-\sigma_\tau u^{(\tau)})
  \cdot\phi dxdt + \int_0^T\int_\Omega\na \phi:
  B(w^{(\tau)},\Phi^{(\tau)})\na w^{(\tau)} dxdt \\
  &\phantom{xx}{}+ \eps\int_0^T\int_\Omega\bigg(\sum_{|\alpha|=m}
  D^\alpha(w^{(\tau)}-w^D)\cdot D^\alpha\phi 
  + (w^{(\tau)}-w^D)\cdot\phi\bigg)dxdt \nonumber \\
  &= \int_0^T\int_\Omega r(u^{(\tau)})\cdot\phi dxdt, \nonumber \\
  \lambda^2&\int_0^T\int_\Omega\big(\ell^2\Delta\Phi^{(\tau)}\Delta\theta
  + \na\Phi^{(\tau)}\cdot\na\theta\big)dxdt
  = \int_0^T\int_\Omega\bigg(\sum_{i=1}^n z_iu_i^{(\tau)} 
  + f\bigg)\theta dxdt \label{2.tau2}
\end{align}
for piecewise constant in time functions $\phi:(0,T)\to X$ and
$\theta:(0,T)\to H_{D,N}^2(\Omega)$, recalling that $X=H^m(\Omega;\R^n)\cap H_D^1(\Omega;\R^n)$. Lemma \ref{lem.est} and the $L^\infty(\Omega)$ estimate of $u_i^k$ imply the uniform bounds
\begin{align}
  \|(u_i^{(\tau)})^{1/2}\|_{L^2(0,T;H^1(\Omega))}
  + \|u_i^{(\tau)}\|_{L^2(0,T;H^1(\Omega))}
  + \|u_i^{(\tau)}\|_{L^\infty(\Omega_T)} &\le C, \label{2.ui} \\
  \|u_0^{(\tau)}\|_{L^2(0,T;H^1(\Omega))}
  + \|u_0^{(\tau)}\|_{L^\infty(\Omega_T)}
  + \|\log u_0^{(\tau)}\|_{L^2(0,T;H^1(\Omega))} &\le C, 
  \label{2.u0} \\
  \sqrt{\eps}\|w_i^{(\tau)}\|_{L^2(0,T;H^m(\Omega))}
  + \|\Phi^{(\tau)}\|_{L^2(0,T;H^2(\Omega))} &\le C, \label{2.Phi}
\end{align}
where $i=1,\ldots,n$. 
We also need a uniform bound for the discrete time derivative.

\begin{lemma}\label{lem.time}
There exists a constant $C>0$ independent of $(\eps,\tau)$ such that
for all $i=1,\ldots,n$,
$$
  \tau^{-1}\|u_i^{(\tau)}-\sigma_\tau u_i^{(\tau)}\|_{L^2(0,T;X')}
  + \tau^{-1}\|u_0^{(\tau)}-\sigma_\tau u_0^{(\tau)}\|_{L^2(0,T;X')}
  \le C.
$$
\end{lemma}

\begin{proof}
Let $\phi:(0,T)\to X$ be piecewise constant. Since
\begin{align*}
  \int_0^T\int_\Omega&\na\phi: B(u^{(\tau)},\Phi^{(\tau)})\na w^{(\tau)}
  dxdt \\
  &= \sum_{i=1}^n D_i\int_0^T\int_\Omega\big(\na u_i^{(\tau)}
  - u_i^{(\tau)}\na\log u_0^{(\tau)} + z_i u_i^{(\tau)}\na\Phi^{(\tau)}
  \big)\cdot\na\phi_i dxdt,
\end{align*}
we find that
\begin{align*}
  \frac{1}{\tau}\bigg|&\int_0^T\int_\Omega(u_i^{(\tau)}-\sigma_\tau
  u_i^{(\tau)})\phi_i dxdt\bigg|
  \le \eps\|w_i^{(\tau)}-w_i^D\|_{L^2(0,T;H^m(\Omega))}
    \|\phi_i\|_{L^2(0,T;H^m(\Omega))} \\
  &\quad + C\big(\|\na u_i^{(\tau)}\|_{L^2(\Omega_T)} 
  + \|\na\log u_0^{(\tau)}\|_{L^2(\Omega_T)}  + \|\na\Phi^{(\tau)}\|_{L^2(\Omega_T)}\big)\|\na\phi_i\|_{L^2(\Omega_T)} \\
  &\phantom{xx}+ \|r_i(u^{(\tau)})\|_{L^2(\Omega_T)}\|\phi_i\|_{L^2(\Omega_T)} \\
  &\le C\|\phi_i\|_{L^2(0,T;H^m(\Omega))}.
\end{align*}
By a density argument, this inequality holds for all 
$\phi_i\in L^2(0,T;X)$, showing the desired bound for the discrete time derivative of $u_i^{(\tau)}$. Summing the bounds over $i=1,\ldots,n$ yields the bound for $u_0^{(\tau)}$. 
\end{proof}

Estimates \eqref{2.ui}--\eqref{2.u0} and Lemma \ref{lem.time}
allow us to apply the Aubin--Lions lemma in the version of
\cite{DrJu12} to conclude the existence of a subsequence, which is not relabeled, such that for $i=1,\ldots,n$, as $(\eps,\tau)\to 0$,
$$
  u_i^{(\tau)}\to u_i, \quad u_0^{(\tau)}\to u_0\quad  
  \mbox{strongly in }L^2(\Omega_T).
$$
In view of the uniform $L^\infty(\Omega_T)$ bound for $u_i^{(\tau)}$ and
$u_0^{(\tau)}$, these convergences hold in $L^p(\Omega_T)$ for all
$p<\infty$. Moreover, by \eqref{2.Phi} and Lemma \ref{lem.time}, 
up to a subsequence,
\begin{align*}
  \eps w_i^{(\tau)}\to 0 &\quad\mbox{strongly in }L^2(0,T;H^m(\Omega)),
  \\ 
  \Phi^{(\tau)}\rightharpoonup \Phi &\quad\mbox{weakly in }
  L^2(0,T;H^2(\Omega)), \\
  \tau^{-1}(u_i^{(\tau)}-\sigma_\tau u_i^{(\tau)})\rightharpoonup
  \pa_t u_i &\quad\mbox{weakly in }L^2(0,T;X'), \ i=1,\ldots,n.
\end{align*}

We claim that $\na\log u_0^{(\tau)}\rightharpoonup\na\log u_0$
weakly in $L^2(\Omega_T)$. It follows from \eqref{2.u0} that (for a
subsequence) $\na\log u_0^{(\tau)}\rightharpoonup v$ weakly in
$L^2(\Omega_T)$. We need to identify $v=\nabla \log u_0$.
We know that (again for a subsequence) $u_0^{(\tau)}\to u_0$
a.e.\ in $\Omega_T$. Therefore $\log u_0^{(\tau)}\to\log u_0$ a.e.\ in $\Omega_T$, since $u_0$ can vanish at most on a set of measure zero. The $L^2(\Omega_T)$ bound for $\log u_0^{(\tau)}$ shows that $\log u_0^{(\tau)}\to\log u_0$ strongly in $L^2(\Omega_T)$.
Hence, we conclude that $v=\na\log u_0$, proving the claim.

These convergences are sufficient to pass to the limit $(\eps,\tau)\to 0$
in \eqref{2.tau1}--\eqref{2.tau2} to find that $(u,\Phi)$ solves
\eqref{1.weak1}--\eqref{1.weak2} for smooth test functions. By a density argument, we may choose test functions from $L^2(0,T;H_D^1(\Omega))$ and $L^2(0,T;H^2_{D,N}(\Omega))$, respectively. The validity of the initial and Dirichlet boundary conditions is shown as in \cite{GeJu18}.
Estimates similar as in the proof of Lemma \ref{lem.time} 
(with $\eps=0$) show that $\pa_t u_i\in L^2(0,T;H^1_D(\Omega)')$
for $i=1,\ldots,n$. Then we conclude from $u_i\in L^2(0,T;H^1(\Omega))$
that $u_i\in C^0([0,T];L^2(\Omega))$. Thus, the initial datum is satisfied in the sense of $L^2(\Omega)$. 

It remains to verify the free energy inequality \eqref{1.fei}
under the assumptions $r_i(u)=0$ and $\log(u_i^D/u_0^D)+z_i\Phi^D=c_i\in\R$ for $i=1,\ldots,n$.
By definition of $w_i^D$, this implies that $\na w_i^D=0$. 
Then \eqref{2.dei} becomes
\begin{align*}
  H(u^k) - H(u^{k-1}) + \tau\int_\Omega\sum_{i=1}^n D_i u_i^k
  \bigg|\na\bigg(\log\frac{u_i^k}{u_0^k}+z_i\Phi^k\bigg)\bigg|^2 dx
  + \eps\tau\|w^k-w^D\|_{H^m(\Omega)}^2 \le 0.
\end{align*}
A summation over $k=j,\ldots,J$ gives
\begin{align}\label{2.feitau}
  H(u^{(\tau)}(t)) - H(u^{(\tau)}(s)) &+ \int_s^t\int_\Omega
  \sum_{i=1}^n D_i u_i^{(\tau)}\bigg|\na\bigg(\log
  \frac{u_i^{(\tau)}}{u_0^{(\tau)}}+z_i\Phi^{(\tau)}\bigg)\bigg|^2
  dxd\sigma \\
  &+ \eps\int_s^t\|w^{(\tau)}-w^D\|_{H^m(\Omega)}^2 d\sigma \le 0,
  \nonumber
\end{align}
where $s\in((j-1)\tau,j\tau]$ and $t\in((J-1)\tau,J\tau]$.
We wish to pass to the limit $(\eps,\tau)\to 0$ in this inequality.

The a.e.\ convergence of $u_i^{(\tau)}$ implies that
$H(u^{(\tau)}(t))\to H(u(t))$ for a.e.\ $t\in(0,T)$ and, since 
$u_i\in C^0([0,T];L^2(\Omega))$, this convergence holds in fact
for all $t\in[0,T]$. Moreover, $\eps(w^{(\tau)}-w^D)\to 0$ strongly $L^2(0,T;H^m(\Omega))$. 
It follows from the strong convergence of $u_i^{(\tau)}$ 
in $L^2(\Omega_T)$ that $(u_i^{(\tau)})^{1/2}\to\sqrt{u_i}$ strongly in
$L^4(\Omega_T)$. Hence, together with the weak convergence of 
$\na\Phi^{(\tau)}$ in $L^2(\Omega_T)$, we have
$$
  (u_i^{(\tau)})^{1/2}\na\Phi^{(\tau)}\rightharpoonup
  \sqrt{u_i}\na\Phi\quad\mbox{weakly in }L^{4/3}(\Omega_T).
$$
Furthermore, since $\na\log u_0^{(\tau)}\rightharpoonup\na\log u_0$
weakly in $L^2(\Omega_T)$,
\begin{align}\label{2.J}
  (u_i^{(\tau)})^{1/2}&\na\log\frac{u_i^{(\tau)}}{u_0^{(\tau)}}
  = 2\na(u_i^{(\tau)})^{1/2} - (u_i^{(\tau)})^{1/2}\na\log u_0^{(\tau)}
  \\
  &\rightharpoonup 2\na\sqrt{u_i} - \sqrt{u_i}\na\log u_0
  =: \sqrt{u_i}\na\log\frac{u_i}{u_0}\quad\mbox{weakly in }
  L^{4/3}(\Omega_T). \nonumber
\end{align}
On the other hand, the sequences $\na(u_i^{(\tau)})^{1/2}$
and $(u_i^{(\tau)})^{1/2}\na\log u_0^{(\tau)}$ are uniformly
bounded in $L^2(\Omega_T)$. Therefore, convergence \eqref{2.J} also
holds in $L^2(\Omega_T)$. Consequently, 
\begin{align*}
  \int_\Omega u_i\bigg|\na\bigg(\log\frac{u_i}{u_0}+z_i\Phi\bigg)
  \bigg|^2 dx
  &= \int_\Omega\big|2\na\sqrt{u_i} - \sqrt{u_i}\na\log u_0
  + \sqrt{u_i}z_i\na\Phi\big|^2 dx \\
  &\le \liminf_{(\eps,\tau)\to 0}
  \int_\Omega u_i^{(\tau)}\bigg|\na\bigg(\log
  \frac{u_i^{(\tau)}}{u_0^{(\tau)}} + z_i\Phi^{(\tau)}\bigg)
  \bigg|^2 dx.
\end{align*}
Then \eqref{1.fei} follows after passing to the limit inferior 
$(\eps,\tau)\to 0$ in
\eqref{2.feitau}, completing the proof of Theorem \ref{thm.ex}.

\begin{remark}\label{rem.feir}\rm
Let the reaction rates $r_i:\overline{\dom}\to\R$ be Lipschitz continuous and quasi-positive, i.e.\ $r_i(u)\ge 0$ for all $u\in\dom$ with $u_i=0$. We assume that the total reaction rate is nonnegative, i.e.\ $\sum_{i=1}^n r_i(u)\le 0$ for all $u\in\dom$, and that $r_i(u)\log u_i=0$ if $u_i=0$. We claim that the free energy inequality becomes
\begin{equation}
  H(u(t)) + \int_s^t\int_\Omega\sum_{i=1}^n D_iu_i|\na w_i|^2 
  dxd\sigma\le H(u(s)) + \int_s^t\int_\Omega\sum_{i=1}^n r_i(u)(w_i-w_i^D)dxd\sigma.
\end{equation}
This inequality follows from \eqref{2.feitau} after including the reaction rates and taking the limit $(\eps,\tau)$ in
\begin{align*}
  \int_s^t&\int_\Omega\sum_{i=1}^n r_i(u^{(\tau)})
  (w_i^{(\tau)}-w_i^D)dxd\sigma \\
  &= \int_s^t\int_\Omega\sum_{i=1}^n r_i(u^{(\tau)})
  \big(\log u_i^{(\tau)} - \log u_0^{(\tau)} + z_i\Phi^{(\tau)}
  - w_i^D\big)dxd\sigma.
\end{align*}

Indeed, the strong limit $u_i^{(\tau)}\to u_i$ in $L^2(\Omega_T)$ shows that $r_i(u^{(\tau)})w_i^D\to r_i(u)w_i^D$ strongly in $L^1(\Omega_T)$ as $(\eps,\tau)\to 0$. Moreover, since $\log u_0^{(\tau)}\to\log u_0$ strongly in $L^2(\Omega)$, we have $r_i(u^{(\tau)})\log u_0^{(\tau)}\to r_i(u)\log u_0$ strongly in $L^1(\Omega_T)$. It remains to show that $r_i(u^{(\tau)})\log u_i^{(\tau)}\to r_i(u)\log u_i$ strongly in $L^1(\Omega_T)$.
We have $r_i(u^{(\tau)})\log u_i^{(\tau)}\to r_i(u)\log u_i$ a.e.\ in $\Omega_T$ if $u_i>0$. If $u_i=0$, by assumption, we have $r_i(u)\log u_i=0$ and therefore $r_i(u^{(\tau)})\log u_i^{(\tau)}\to r_i(u)\log u_i$ a.e.\ in $\Omega_T$ as well. Moreover, $r_i(u)\log u_i$ is bounded. Hence, by dominated convergence, $r_i(u^{(\tau)})\log u_i^{(\tau)}\to r_i(u)\log u_i$ strongly in $L^1(\Omega_T)$, and the claim follows.
\qed\end{remark}


\section{Proof of Theorem \ref{thm.wsu}}\label{sec.wsu}

Let $(u,\Phi)$ be a weak solution and $(\bar{u},\bar\Phi)$ be a strong solution to \eqref{1.mass}--\eqref{1.bcphi}. In this section, we interpret $H(u)$ and $H(\bar{u})$ as functionals depending on $u=(u_0,\ldots,u_n)$ and $\bar{u}=(\bar{u}_0,\ldots,\bar{u}_n)$. This notation is only needed to determine the variational derivative of $H$ and will not lead to any confusion in the following computations.
We split the lengthy proof in several steps.

{\em Step 1: Calculation of the time derivative of 
$H(u,\Phi|\bar{u},\bar\Phi)$.} In the following, we write
\begin{align*}
  & H(u,\Phi|\bar{u},\bar{\Phi}) = H_1(u|\bar{u}) 
  + H_2(\Phi|\bar{\Phi}), \quad\mbox{where} \\
  & H_1(u|\bar{u}) = H_1(u) - H_1(\bar{u}) - H'_1(\bar{u})(u-\bar{u}), \\
  & H_2(\Phi|\bar{\Phi}) = H_2(\Phi) - H_2(\bar\Phi)
  - H'_2(\bar\Phi)(\Phi-\bar\Phi),
\end{align*}
where $H_1(u)=\int_\Omega h_1(u)dx$ with
$h_1(u)=\sum_{i=0}^n\int_{u_i^D}^{u_i}\log(s/u_i^D)ds$,
$H_2(\Phi)= \frac12\lambda^2\int_\Omega
(\ell^2|\Delta(\Phi-\Phi^D)|^2+|\na(\Phi-\Phi^D)|^2)dx$,
and $H'_1(\bar{u})(u-\bar{u})$ is the variational derivative of $H_1$ at $\bar{u}$ in the direction of $u-\bar{u}$ (similarly for $H'_2(\bar\Phi)(\Phi-\bar\Phi)$). 
We compute the time derivative of $H_1(u|\bar{u})$, split the sum over $i=0,\ldots,n$ into $i=0$ and the sum over $i=1,\ldots,n$, and insert $\pa_t u_0=-\sum_{i=1}^n\pa_t u_i$, 
$\pa_t \bar{u}_0=-\sum_{i=1}^n\pa_t \bar{u}_i$:
\begin{align*}
  \frac{dH_1}{dt}(u|\bar{u}) &= \frac{dH_1}{dt}(u) 
  - \frac{dH_1}{dt}(\bar{u}) - \frac{d}{dt}\int_\Omega\sum_{i=0}^n
  \frac{\pa h_1}{\pa u_i}(\bar{u})(u_i-\bar{u}_i)dx \\
  &= \frac{dH_1}{dt}(u) 
  - \sum_{i=0}^n\bigg(\bigg\langle\pa_t u_i,\frac{\pa h_1}{\pa u_i}(\bar{u})\bigg\rangle
  + \bigg\langle\pa_t\bar{u}_i,\frac{u_i}{\bar{u}_i}-1
  \bigg\rangle\bigg) \\
  &= \frac{dH_1}{dt}(u) - \sum_{i=1}^n\bigg(\bigg\langle\pa_t u_i,\frac{\pa h_1}{\pa u_i}(\bar{u})
  -\frac{\pa h_1}{\pa u_0}(\bar{u})\bigg\rangle
  + \bigg\langle\pa_t\bar{u}_i,\frac{u_i}{\bar{u}_i}
  -\frac{u_0}{\bar{u}_0}\bigg\rangle\bigg).
\end{align*}
Next, we insert equation \eqref{1.mass} for $u_i$ and $\bar{u}_i$ and use $(\pa h_1/\pa u_i)(\bar{u})=\log(\bar{u}_i/u_i^D)$:
\begin{align*}
  \frac{dH_1}{dt}(u|\bar{u}) &= \frac{dH_1}{dt}(u) 
  + \int_\Omega\sum_{i=1}^n D_iu_i\na w_i\cdot\na\bigg(\log
  \frac{\bar{u}_i}{\bar{u}_0} 
  - \log\frac{u_i^D}{u_0^D}dx\bigg)dx \\
  &\phantom{xx}+ \int_\Omega\sum_{i=1}^n D_i\bar{u}_i\na \bar{w}_i\cdot
  \bigg(\frac{u_i}{\bar{u}_i}\na\log\frac{u_i}{\bar{u}_i}
  - \frac{u_0}{\bar{u}_0}\na\log\frac{u_0}{\bar{u}_0}\bigg)dx.
\end{align*}
A similar computation for $H_2(\Phi|\bar\Phi)$ leads to
\begin{align*}
  \frac{dH_2}{dt}&(\Phi|\bar{\Phi}) 
  = \lambda^2 \big\langle
  (\ell^2\Delta-1)\Delta\pa_t(\Phi-\bar{\Phi}),\Phi-\bar{\Phi}\big\rangle
  = \sum_{i=1}^n\langle z_i\pa_t(u_i-\bar{u}_i),\Phi-\bar{\Phi}\rangle\\
  &= -\int_\Omega\sum_{i=1}^n D_iz_i(u_i\na w_i-\bar{u}_i\na\bar{w}_i)
  \cdot\na(\Phi-\bar{\Phi})dx \\
  &= \frac{dH_2}{dt}(\Phi) 
  + \int_\Omega\sum_{i=1}^n D_iz_i\big(u_i\na w_i\cdot\na\bar\Phi
  - u_i\na w_i\cdot\na\Phi^D
  + \bar{u}_i\na\bar{w}_i\cdot\na(\Phi-\bar\Phi)\big) dx,
\end{align*}
where we abbreviated $\bar{w}_i=\log(\bar{u}_i/\bar{u}_0)+z_i\bar\Phi$. As $u_i$ is only nonnegative, the expression $\na\log u_i$ may be not integrable. Therefore, we define $\na\log(u_i/u_0)
:= (2\na\sqrt{u_i}-\sqrt{u}_i\na\log u_0)/\sqrt{u_i}$ if
$u_i>0$ and $\na\log(u_i/u_0):=0$ else. This expression may be not integrable but $\sqrt{u_i}\na\log(u_i/u_0)$ lies in $L^2(\Omega_T)$ since $\na\log u_0\in L^2(\Omega_T)$. Thus, the expression $\sqrt{u_i}\na w_i=u_i\na\log(u_i/u_0)+u_iz_i\na\Phi\in L^2(\Omega_T)$ is well defined. In a similar way, we define $\na\log(u_i/\bar{u}_i)$, which is possible since $\bar{u}_i$ is strictly positive, and we have $\sqrt{u_i}\na\log(u_i/\bar{u}_i)\in L^2(\Omega_T)$.

We insert the free energy inequality \eqref{1.fei}, namely 
$$
  \frac{dH_1}{dt}(u) + \frac{dH_2}{dt}(\Phi)
  \le -\int_\Omega\sum_{i=1}^n D_iu_i|\na w_i|^2 dx,
$$
and rearrange the terms,
\begin{align}\label{3.Huu}
  \frac{dH}{dt}&(u,\Phi|\bar{u},\bar\Phi) = \frac{dH_1}{dt}(u|\bar{u})
  + \frac{dH_2}{dt}(\Phi|\bar\Phi) \\
  &\le -\int_\Omega\sum_{i=1}^n D_iu_i\na w_i\cdot\bigg(
  \na\log\frac{u_i}{\bar{u}_i} - \na\log\frac{u_0}{\bar{u}_0}
  + z_i\na(\Phi-\bar\Phi)\bigg)dx \nonumber \\
  &\phantom{xx}+ \int_\Omega\sum_{i=1}^n D_i\bar{u}_i\na\bar{w}_i\cdot\bigg(
  \frac{u_i}{\bar{u}_i}\na\log\frac{u_i}{\bar{u}_i}
  - \frac{u_0}{\bar{u}_0}\na\log\frac{u_0}{\bar{u}_0}
  + z_i\na(\Phi-\bar\Phi)\bigg)dx. \nonumber 
\end{align}
At this point, we observe that the terms involving $\na w_i^D$ cancel even if $\na w_i^D$ does not vanish, since they also appear in the free energy inequality \eqref{1.fei}.

The terms involving the solvent concentrations $u_0$ and $\bar{u}_0$
can be integrated into the sum over $i$ if we interpret system
\eqref{1.mass} as equations for $u_0,u_1,\ldots,u_n$. For this, we
observe that $u_0$ solves
$$
  \pa_t u_0 = -\diver\sum_{i=1}^n D_iu_i\na w_i
  = -\diver\bigg\{\sum_{i=1}^n D_iu_i\na\log\frac{u_i}{u_0}
  + \bigg(\sum_{i=1}^n D_iz_iu_i\bigg)\na\Phi\bigg\}.
$$
Then \eqref{1.mass} reads as
\begin{align}\label{3.AP}
  & \pa_t u_i = \diver\sum_{j=0}^n(A_{ij}\na\log u_j 
  + Q_{ij}\na\Phi), \quad\mbox{where }i=0,\ldots,n, \\
  & A = (A_{ij}) = \begin{pmatrix}
  \sum_{i=1}^n D_iu_i & -D_1u_1 & \cdots & -D_nu_n \\
  -D_1u_1 & D_1u_1 & & 0 \\
  \vdots & 0 & \ddots & 0 \\
  -D_nu_n & 0 & & D_nu_n \end{pmatrix}, \nonumber \\
  & Q = (Q_{ij}) = \begin{pmatrix}
  -\sum_{i=1}^n D_iz_iu_i & 0 & \cdots & 0 \\
  0 & D_1z_1u_1 & & 0 \\
  \vdots & 0 & \ddots & 0 \\
  0 & 0 & & D_nz_nu_n \end{pmatrix} \nonumber 
\end{align}
setting $z_0:=0$. We define in a similar way $\bar{A}$ and $\bar{Q}$. 
With this notation, \eqref{3.Huu} becomes
\begin{align*}
  \frac{dH}{dt}(u,\Phi|\bar{u},\bar\Phi) 
  &\le -\int_\Omega\sum_{i,j=0}^n(A_{ij}\na\log u_j+Q_{ij}\na\Phi)
  \cdot\bigg(\na\log\frac{u_i}{\bar{u}_i} + z_i\na(\Phi-\bar\Phi)
  \bigg)dx \\
  &\phantom{xx}{}+ \int_\Omega\sum_{i,j=0}^n(\bar{A}_{ij}
  \na\log\bar{u}_j+\bar{Q}_{ij}\na\bar\Phi)
  \cdot\bigg(\frac{u_i}{\bar{u}_i}\na\log\frac{u_i}{\bar{u}_i} + z_i\na(\Phi-\bar\Phi)\bigg)dx.
\end{align*}
We add and subtract some terms and integrate over $(0,t)$:
\begin{align}\label{3.H}
  & H((u,\Phi)(t)|(\bar{u},\bar\Phi)(t)) 
  - H((u,\Phi)(0)|(\bar{u},\bar\Phi)(0))
  \le I_1+I_2+I_3, \quad\mbox{where} \\
  & I_1 = -\int_0^t\int_\Omega\sum_{i,j=0}^n
  \bigg(A_{ij}\na\log\frac{u_j}{\bar{u}_j}
  + Q_{ij}\na(\Phi-\bar\Phi)\bigg)\cdot
  \bigg(\na\log\frac{u_i}{\bar{u}_i} + z_i\na(\Phi-\bar\Phi)
  \bigg)dxds, \nonumber \\
  & I_2 = -\int_0^t\int_\Omega\sum_{i,j=0}^n
  u_i\bigg\{\bigg(\frac{A_{ij}}{u_i}
  - \frac{\bar{A}_{ij}}{\bar{u}_i}\bigg)\na\log\bar{u}_j
  + \bigg(\frac{Q_{ij}}{u_i}-\frac{\bar{Q}_{ij}}{\bar{u}_i}\bigg)
  \na\bar\Phi\bigg\}\cdot\na\log\frac{u_i}{\bar{u}_i}dxds, \nonumber \\
  & I_3 = -\int_0^t\int_\Omega\sum_{i,j=0}^n\big((A_{ij}-\bar{A}_{ij})
  \na\log\bar{u}_j + (Q_{ij}-\bar{Q}_{ij})\na\bar\Phi\big)
  \cdot z_i\na(\Phi-\bar\Phi)dxds. \nonumber 
\end{align}
Observe that $u(0)=\bar{u}(0)$, implying that
$H((u,\Phi)(0)|(\bar{u},\bar\Phi)(0))=0$.

{\em Step 2: Estimation of $I_3$.}
By Young's inequality, we have
\begin{equation}\label{3.I3}
  I_3\le C\int_0^t\sum_{i=1}^n\big(\|u_i-\bar{u}_i\|_{L^2(\Omega)}^2
  + \|\na(\Phi-\bar\Phi)\|_{L^2(\Omega)}^2\big)ds,
\end{equation}
where $C>0$ depends on the $L^\infty(\Omega_T)$ norms of 
$\na\log\bar{u}_j$ and $\na\bar\Phi$. 

The treatment of $I_1$ and $I_2$ is more delicate.

{\em Step 3: Estimation of $I_1$.} We write $I_1=I_{11}+I_{12}+I_{13}$,
where
\begin{align*}
  I_{11} &= -\int_0^t\int_\Omega\sum_{i,j=0}^n A_{ij}
  \na\log\frac{u_j}{\bar{u}_j}\cdot\na\log\frac{u_i}{\bar{u}_i}dxds, \\
  I_{12} &= -\int_0^t\int_\Omega\sum_{i,j=0}^n z_i Q_{ij}
  |\na(\Phi-\bar\Phi)|^2dxds, \\
  I_{13} &= -\int_0^t\int_\Omega\sum_{i,j=0}^n z_i A_{ij}
  \na\log\frac{u_j}{\bar{u}_j}\cdot\na(\Phi-\bar\Phi)dxds \\
  &\phantom{xx}
  -\int_0^t\int_\Omega\sum_{i,j=0}^n Q_{ij}\na\log\frac{u_i}{\bar{u}_i}
  \cdot\na(\Phi-\bar\Phi)dxds.
\end{align*}
It follows from $0\le u_i\le 1$ that $|Q_{ij}|\le C$ and consequently
$$
  I_{12} \le C\int_0^t\|\na(\Phi-\bar\Phi)\|_{L^2(\Omega)}^2 ds.
$$
The matrix $A$ is not positive definite since $u_i=0$ is possible. We show that the matrix $G$, defined by $G_{ij}=A_{ij}/\sqrt{u_iu_j}$, is positive definite on the subspace $L=\{Y\in\R^{n+1}:
\sum_{i=0}^n\sqrt{u_i}Y_i=0\}$. The corresponding positive bound will help us to estimate $I_{13}$ and $I_2$. We introduce the projections
$$
  (P_LY)_i = Y_i - \sqrt{u_i}\sum_{j=0}^n\sqrt{u_j}Y_j, \quad 
  (P_{L^\perp}Y)_i =  \sqrt{u_i}\sum_{j=0}^n\sqrt{u_j}Y_j,
$$
for all $i=0,\ldots,n$ and $Y\in\R^{n+1}$. 

\begin{lemma}
Let $Y_i=\sqrt{u_i}\na\log(u_i/\bar{u}_i)\in L^2(\Omega_T)$ for $i=0,\ldots,n$. Then
$$
  I_{11} \le -D_*\int_0^t\int_\Omega\bigg(\frac{|(P_LY)_0|^2}{u_0}
  + \sum_{i=1}^n|(P_LY)_i|^2\bigg)dxds,
$$
where $D_*=\min_{i=1,\ldots,n}D_i>0$.
\end{lemma}

\begin{proof}
Let $G_{ij}=A_{ij}/\sqrt{u_iu_j}$ if $u_iu_j>0$ and $G_{ij}=0$ else.
Recall that by definition, $\na\log(u_i/\bar{u}_i)=(2\na\sqrt{u_i}
-\sqrt{u_i}\na\log\bar{u}_i)/\sqrt{u_i}=Y_i/\sqrt{u_i}$ 
if $u_i>0$. In this case,
$$
  A_{ij}\na\log\frac{u_i}{\bar{u}_i}\cdot\na\log\frac{u_j}{\bar{u}_j}
  = G_{ij} Y_i Y_j.
$$
If $u_i=0$ or $u_j=0$, either $Y_i=0$ or $Y_j=0$ and hence,
the previous expression vanishes. Therefore,
$$
  I_{11} = -\int_0^t\int_\Omega\sum_{i,j=0}^n G_{ij}Y_iY_j dxds.
$$

We know that $u_0>0$ a.e.\ in $\Omega_T$. A straightforward computation shows that $\operatorname{ran}G=L$,
implying that $\operatorname{ker}G=L^\perp$. 
We claim that the matrix $G$ satisfies for all $Y\in\R^{n+1}$,
\begin{equation}\label{3.PGP}
  (P_LY)^T G(P_LY) \ge \frac{|(P_LY)_0|^2}{u_0}
  + \sum_{i=1}^n |(P_LY)_i|^2.
\end{equation}
Indeed, introduce first the matrix 
$$
  G_* = D_*\begin{pmatrix}
  u_0^{-1}\sum_{i=1}^n u_i & -\sqrt{u_1/u_0} & \cdots & -\sqrt{u_n/u_0}\\
  -\sqrt{u_1/u_0} & 1 & & 0 \\
  \vdots & 0 & \ddots & 0 \\
  -\sqrt{u_n/u_0} & 0 & & 1 \end{pmatrix}.
$$
Let $Y\in\R^{n+1}$ and $z_i:=(P_LY)_i=Y_i-\sqrt{u_i}\sum_{j=0}^n
\sqrt{u_j}Y_j$. Then, because of $\sum_{i=0}^n\sqrt{u_i}z_i=0$,
\begin{align*}
  z^TG_*z &= D_*\frac{z_0^2}{u_0}\sum_{i=1}^n u_i
  - 2D_*\frac{z_0}{\sqrt{u_0}}\sum_{i=1}^n\sqrt{u_i}z_i
  + D_*\sum_{i=1}^n z_i^2 \\
  &= D_*\frac{z_0^2}{u_0}(1-u_0) + 2D_*\frac{z_0}{\sqrt{u_0}}
  \sqrt{u_0}z_0 + D_*\sum_{i=1}^n z_i^2 \\
  &= D_*\bigg\{\bigg(\frac{1}{u_0}+1\bigg)z_0^2
  + \sum_{i=1}^n z_i^2\bigg\}
  \ge D_*\bigg(\frac{|(P_LY)_0|^2}{u_0} + \sum_{i=1}^n|(P_LY)_i|^2\bigg).
\end{align*}
This implies that
\begin{align*}
  z^T(G-G_*)z &= \frac{z_0^2}{u_0}\sum_{i=1}^n(D_i-D_*)u_i
  -2\frac{z_0}{\sqrt{u_0}}\sum_{i=1}^n(D_i-D_*)\sqrt{u_i}z_i
  + \sum_{i=1}^n(D_i-D_*)z_i^2 \\
  &= \sum_{i=1}^n(D_i-D_*)\bigg(\frac{z_0}{\sqrt{u_0}}\sqrt{u_i}
  + z_i\bigg)^2 \ge 0,
\end{align*}
and we infer that $(P_LY)^TG(P_LY)\ge (P_LY)^TG_*(P_LY)$, proving
claim \eqref{3.PGP}. 

We choose now $Y_i=\sqrt{u_i}\na\log(u_i/\bar{u}_i)$. The expression
$$
  \frac{|(P_LY)_0|^2}{u_0} = \bigg|\na\log\frac{u_0}{\bar{u}_0}
  - \sum_{j=0}^n\sqrt{u_j}Y_j\bigg|^2
$$
is integrable in $\Omega_T$ since $\na\log u_0\in L^2(\Omega_T)$,
and $\sqrt{u_j}Y_j\in L^2(\Omega_T)$. 
Therefore, we can integrate inequality \eqref{3.PGP}, and 
in view of $GY = G(P_LY)$ 
(since $\operatorname{ker}G=L^\perp$), this shows that
\begin{align*}
  I_{11} &= -\int_0^t\int_\Omega Y^TGY dxds
  = -\int_0^t\int_\Omega (P_LY)^T G(P_LY)dxds \\
  &\le -D_*\int_0^t\int_\Omega \bigg(\frac{|(P_LY)_0|^2}{u_0}
  + \sum_{i=1}^n|(P_LY)_i|^2\bigg)dxds,
\end{align*}
which finishes the proof.
\end{proof}

\begin{lemma}
Let $Y_i=\sqrt{u_i}\na\log(u_i/\bar{u}_i$) for $i=0,\ldots,n$. For any $\eps>0$, there exists $C(\eps)>0$ such that
$$
  I_{13} \le \eps\int_0^t\int_\Omega\bigg(\frac{|(P_LY)_0|^2}{u_0}
  + \sum_{i=1}^n|(P_LY)_i|^2\bigg)dxds
  + C(\eps)\int_0^t\|\na(\Phi-\bar\Phi)\|_{L^2(\Omega)}^2 ds.
$$
\end{lemma}

\begin{proof}
We take into account the structures of the matrices $A$ and $Q$:
\begin{align*}
  I_{13} &= -\int_0^t\int_\Omega\sum_{i=1}^n z_i\bigg(A_{i0}
  \na\log\frac{u_0}{\bar{u}_0} + A_{ii}\na\log\frac{u_i}{\bar{u}_i}\bigg)
  \cdot\na(\Phi-\bar\Phi)dxds \\
  &\phantom{xx}
  - \int_0^t\int_\Omega\bigg(Q_{00}\na\log\frac{u_0}{\bar{u}_0}
  + \sum_{i=1}^n Q_{ii}\na\log\frac{u_i}{\bar{u}_i}\bigg)
  \cdot\na(\Phi-\bar\Phi)dxds. 
\end{align*}
Since $Q_{00}=-\sum_{i=1}^n D_iz_iu_i$ and $Q_{ii}=D_iz_iu_i$, 
we have
$$
  Q_{00}\na\log\frac{u_0}{\bar{u}_0}
  + \sum_{i=1}^n Q_{ii}\na\log\frac{u_i}{\bar{u}_i}
  = -\sum_{i=1}^n D_iz_iu_i\na\bigg(\log\frac{u_0}{\bar{u}_0}
  - \log\frac{u_i}{\bar{u}_i}\bigg)dx.  
$$
Furthermore, because of $A_{i0}=-D_iu_i$ and 
$A_{ii}=D_iu_i$,
$$
  \sum_{i=1}^n z_i \bigg(A_{i0}\na\log\frac{u_0}{\bar{u}_0} 
  + A_{ii}\na\log\frac{u_i}{\bar{u}_i}\bigg)
  = -\sum_{i=1}^n D_iz_iu_i\na\bigg(\log\frac{u_0}{\bar{u}_0}
  - \log\frac{u_i}{\bar{u}_i}\bigg)dx.
$$
This gives
\begin{align*}
  I_{13} &= 2\int_0^t\int_\Omega\sum_{i=1}^n D_iz_iu_i\na\bigg(
  \log\frac{u_0}{\bar{u}_0} - \log\frac{u_i}{\bar{u}_i}\bigg)
  \cdot\na(\Phi-\bar\Phi)dxds \\
  &= 2\int_0^t\int_\Omega\sum_{i=1}^n D_iz_i\bigg(
  u_i\frac{Y_0}{\sqrt{u_0}} - \sqrt{u_i}Y_i\bigg)
  \cdot\na(\Phi-\bar\Phi)dxds.
\end{align*}

Next, we calculate for $i=0,\ldots,n$,
\begin{align}\label{3.PLperp}
  (P_{L^\perp}Y)_i &= \sqrt{u_i}\sum_{j=0}^n u_j
  \na\log\frac{u_j}{\bar{u}_j}
  = \sqrt{u_i}\sum_{j=0}^n(\na u_j - u_j\na\log\bar{u}_j) \\
  &= -\sqrt{u_i}\sum_{j=0}^n u_j\na\log\bar{u}_j
  = \sqrt{u_i}\sum_{j=0}^n(\bar{u}_j-u_j)\na\log\bar{u}_j, \nonumber
\end{align}
where we used $\sum_{j=0}^n\na u_j=0$ and $\sum_{j=0}^n\bar{u}_j\na\log\bar{u}_j=0$. Hence, 
\begin{align*}
  u_i\frac{(P_{L^\perp}Y)_0}{\sqrt{u_0}} - \sqrt{u_i}(P_{L^\perp}Y)_i
  &= \frac{u_i}{\sqrt{u_0}}\bigg(\sqrt{u_0}
  \sum_{j=0}^n(\bar{u}_j-u_j)\na\log\bar{u}_j\bigg) \\
  &\phantom{xx}- \sqrt{u_i}\bigg(\sqrt{u_i}\sum_{j=0}^n(\bar{u}_j-u_j)
  \na\log\bar{u}_j\bigg) = 0. 
\end{align*}
We split $Y_i=(P_LY)_i + (P_{L^\perp}Y)_i$ in $I_{13}$, which leads to
$$
  I_{13} = 2\int_0^t\int_\Omega\sum_{i=1}^n D_iz_i
  \bigg(u_i\frac{(P_LY)_0}{\sqrt{u_0}} - \sqrt{u_i}(P_LY)_i\bigg)
  \cdot\na(\Phi-\bar\Phi)dxds.
$$
An application of Young's lemma finishes the proof.
\end{proof}

The previous lemmas show that
\begin{equation}\label{3.I1}
  I_1 \le (\eps-D_*)\int_0^t\int_\Omega\bigg(\frac{|(P_LY)_0|^2}{u_0}
    + \sum_{i=1}^n|(P_LY)_i|^2\bigg)dxds
    + C(\eps)\int_0^t\|\na(\Phi-\bar\Phi)\|_{L^2(\Omega)}^2 ds.
\end{equation}

{\em Step 4: Estimation of $I_2$.}
We split $I_2=I_{21}+I_{22}$, where
\begin{align}
  I_{21} &= -\int_0^t\int_\Omega\sum_{i,j=0}^n u_i\bigg(\frac{A_{ij}}{u_i}-\frac{\bar{A}_{ij}}{\bar{u}_i}\bigg)
  \na\log\bar{u}_j\cdot\na\log\frac{u_i}{\bar{u}_i}dxds, \nonumber \\
  I_{22} &= -\int_0^t\int_\Omega\sum_{i,j=0}^n u_i
  \bigg(\frac{Q_{ij}}{u_i}-\frac{\bar{Q}_{ij}}{\bar{u}_i}\bigg)
  \na\bar\Phi\cdot\na\log\frac{u_i}{\bar{u}_i}dxds. \label{3.defI22}
\end{align}

\begin{lemma}
For any $\eps>0$, there exists $C(\eps)>0$ such that
$$
  I_{21} \le \eps\int_0^t\int_\Omega\frac{|(P_LY)_0|^2}{u_0}dxds
    + C(\eps)\int_0^t\sum_{i=0}^n\|u_i-\bar{u}_i\|_{L^2(\Omega)}^2 ds.
$$
\end{lemma}

\begin{proof}
Recalling that $Y_i=\sqrt{u_i}\na\log(u_i/\bar{u}_i)$,
we reformulate $I_{21}$ as
\begin{align*}
  I_{21} = -\int_0^t\int_\Omega\sum_{i,j=0}^n u_i
  \bigg(\frac{A_{ij}}{u_i}-\frac{\bar{A}_{ij}}{\bar{u}_i}\bigg)
  \frac{Y_i}{\sqrt{u_i}}\cdot\na\log\bar{u}_j dxds.
\end{align*}
All rows of the matrix $(A_{ij}/u_i-\bar{A}_{ij}/\bar{u}_i)$ 
vanish except the first one, 
\begin{align*}
  \frac{A_{00}}{u_0}-\frac{\bar{A}_{00}}{\bar{u}_0}
  = \sum_{i=1}^n D_i\bigg(\frac{u_i}{u_0}-\frac{\bar{u}_i}{\bar{u}_0}
  \bigg), \quad
  \frac{A_{0j}}{u_0}-\frac{\bar{A}_{0j}}{\bar{u}_0}
  = -D_i\bigg(\frac{u_j}{u_0}-\frac{\bar{u}_j}{\bar{u}_0}
    \bigg)\ \mbox{for }j=1,\ldots,n.
\end{align*}
This shows that
\begin{align}\label{3.auxI21}
  I_{21} &= -\int_0^t\int_\Omega\sum_{j=0}^n u_0
  \bigg(\frac{A_{0j}}{u_0} - \frac{\bar{A}_{0j}}{\bar{u}_0}\bigg)
  \frac{Y_0}{\sqrt{u_0}}\cdot\na\log\bar{u}_j dxds \\
  &= -\int_0^t\int_\Omega M\cdot\bigg(\frac{(P_LY)_0}{\sqrt{u_0}}
  + \frac{(P_{L^\perp}Y)_0}{\sqrt{u_0}}\bigg)dxds, \nonumber
\end{align}
where
\begin{align*}
  M &= \sum_{j=0}^n\bigg(A_{0j} - \frac{u_0}{\bar{u}_0}\bar{A}_{0j}\bigg)
  \na\log\bar{u}_j \\
  &= \sum_{i=1}^n D_i\bigg(u_i-\frac{u_0}{\bar{u}_0}\bar{u}_i\bigg)
  \na\log\bar{u}_0 
  - \sum_{i=1}^n D_i\bigg(u_i-\frac{u_0}{\bar{u}_0}\bar{u}_i\bigg)
  \na\log\bar{u}_i \\
  &= \sum_{i=1}^n D_i(u_i-\bar{u}_i)\na\log\bar{u}_0
  + \bigg(1-\frac{u_0}{\bar{u}_0}\bigg)\sum_{i=1}^n D_i\bar{u}_i\na\log\bar{u}_0 \\
  &\phantom{xx}- \sum_{i=1}^n D_i(u_i-\bar{u}_i)\na\log\bar{u}_i
  - \bigg(1-\frac{u_0}{\bar{u}_0}\bigg)\sum_{i=1}^n D_i\bar{u}_i\na\log\bar{u}_i \\
  &= \sum_{i=1}^n D_i(u_i-\bar{u}_i)\na\log\frac{\bar{u}_0}{\bar{u}_i}
  + (u_0-\bar{u}_0)\sum_{i=1}^n D_i\frac{\bar{u}_i}{\bar{u}_0}
  \na\log\frac{\bar{u}_0}{\bar{u}_i}.
\end{align*}
Since $\na\log\bar{u}_i$ is bounded in $L^\infty(\Omega_T)$, we can bound the first term in $I_{21}$:
\begin{align}\label{3.I21a}
  -\int_0^t\int_\Omega & M\cdot\frac{(P_LY)_0}{\sqrt{u_0}}dxds
  \le C\int_0^t\int_\Omega\sum_{i=0}^n|u_i-\bar{u}_i|
  \frac{|(P_LY)_0|}{\sqrt{u_0}}dxds \\
  &\le \eps\int_0^t\int_\Omega\frac{|(P_LY)_0|^2}{u_0}dxds
  + C(\eps)\int_0^t\sum_{i=0}^n\|u_i-\bar{u}_i\|_{L^2(\Omega)}^2 ds,
  \nonumber
\end{align}
where $\eps>0$ is arbitrary.
To estimate the second term in $I_{21}$, we use \eqref{3.PLperp}
and the elementary inequality 
$(\sum_{i=0}^n|u_i-\bar{u}_i|)^2\le(n+1)\sum_{i=0}^n|u_i-\bar{u}_i|^2$:
\begin{align}\label{3.I21b}
  -\int_0^t\int_\Omega M\cdot\frac{(P_{L^\perp}Y)_0}{\sqrt{u_0}}dxds
  &\le C\int_0^t\int_\Omega\sum_{i=0}^n|u_i-\bar{u}_i|
  \sum_{j=0}^n|\bar{u}_j-u_j| dxds \\
  &\le C(n+1)\int_0^t\int_\Omega\sum_{i=0}^n|u_i-\bar{u}_i|^2 dxds.
  \nonumber
\end{align}
The lemma follows after inserting \eqref{3.I21a} and \eqref{3.I21b}
into \eqref{3.auxI21}.
\end{proof}

\begin{lemma}
For any $\eps>0$, there exists $C(\eps)>0$ such that
$$
  I_{22} \le \eps\int_0^t\int_\Omega\frac{|(P_LY)_0|^2}{u_0}dxds
  + C(\eps)\int_0^t\sum_{i=0}^n\|u_i-\bar{u}_i\|_{L^2(\Omega)}^2,
$$
recalling definition \eqref{3.defI22} of $I_{22}$.
\end{lemma}

\begin{proof}
All entries of the matrix $(Q_{ij}/u_i-\bar{Q}_{ij}/\bar{u}_i)$ 
vanish except the element $Q_{00}/u_0-\bar{Q}_{00}/\bar{u}_0
= -\sum_{i=1}^n D_iz_i(u_i/u_0-\bar{u}_i/\bar{u}_0)$. This leads to
\begin{align*}
  I_{22} &= \int_0^t\int_\Omega\sum_{i=1}^n D_iz_iu_0
  \bigg(\frac{u_i}{u_0}-\frac{\bar{u}_i}{\bar{u}_0}\bigg)
  \na\bar\Phi\cdot\na\log\frac{u_0}{\bar{u}_0}dxds \\
  &= \int_0^t\int_\Omega\sum_{i=1}^n D_iz_i\bigg((u_i-\bar{u}_i)
  + \frac{\bar{u}_i}{\bar{u}_0}(\bar{u}_0-u_0)\bigg)
  \na\bar\Phi\cdot\frac{Y_0}{\sqrt{u_0}}dxds \\
  &= \int_0^t\int_\Omega\sum_{i=1}^n D_iz_i\bigg((u_i-\bar{u}_i)
  + \frac{\bar{u}_i}{\bar{u}_0}(\bar{u}_0-u_0)\bigg)
  \bigg(\frac{(P_LY)_0}{\sqrt{u_0}} 
  + \frac{(P_{L^\perp}Y)_0}{\sqrt{u_0}}\bigg)\cdot\na\bar\Phi dxds \\
  &\le C\int_0^t\int_\Omega\sum_{j=0}^n|u_j-\bar{u}_j|
  \bigg(\frac{|(P_LY)_0|}{\sqrt{u_0}} 
  + \frac{|(P_{L^\perp}Y)_0|}{\sqrt{u_0}}\bigg)|\na\bar\Phi|dsdx.
\end{align*}
It follows from \eqref{3.PLperp} that
$$
  \frac{|(P_{L^\perp}Y)_0|}{\sqrt{u_0}}
  = \bigg|\sum_{j=0}(\bar{u}_j-u_j)\na\log\bar{u}_j\bigg|
  \le C\sum_{j=0}^n|\bar{u}_j-u_j|.
$$
Hence, Young's inequality completes the proof.
\end{proof}

We conclude that
\begin{equation}\label{3.I2}
  I_2 \le 2\eps\int_0^t\int_\Omega\frac{|(P_LY)_0|^2}{u_0}dxds
    + C(\eps)\int_0^t\sum_{i=0}^n\|u_i-\bar{u}_i\|_{l^2(\Omega)}^2 ds.
\end{equation}

{\em Step 5: End of the proof.}
We collect \eqref{3.I3}, \eqref{3.I1}, and \eqref{3.I2}:
\begin{align*}
  I_1+I_2+I_3 &\le (3\eps-D_*)\int_0^t\int_\Omega\bigg(
  \frac{|(P_LY)_0|^2}{u_0} + \sum_{i=1}^n|(P_LY)_i|^2\bigg)dxds \\
  &\phantom{xx}+ C(\eps)\int_0^t\bigg(\sum_{i=0}^n\|u_i-\bar{u}_i\|_{L^2(\Omega)}^2
  + \|\na(\Phi-\bar\Phi)\|_{L^2(\Omega)}^2\bigg)ds.
\end{align*}
Thus, choosing $\eps\le D_*/3$, we conclude from \eqref{3.H} that
\begin{align}\label{3.Haux}
  H((u,\Phi)(t)|(\bar{u},\bar\Phi)(t))
  \le C\int_0^t\bigg(\sum_{i=0}^n\|u_i-\bar{u}_i\|_{L^2(\Omega_T)}^2
  + \|\na(\Phi-\bar\Phi)\|_{L^2(\Omega)}^2\bigg)ds.
\end{align}
It follows from \cite[Lemma 16]{HJT22} that
$$
  \sum_{i=0}^n\int_\Omega u_i\log\frac{u_i}{\bar{u}_i}dx
  \ge \frac12\sum_{i=0}^n\int_\Omega(u_i-\bar{u}_i)^2 dx,
$$
and hence,
$$
  2H(u,\Phi|\bar{u},\bar{\Phi})
  \ge \sum_{i=0}^n\|u_i-\bar{u}_i\|_{L^2(\Omega)}^2
  + \lambda^2\|\na(\Phi-\bar{\Phi})\|_{L^2(\Omega)}^2.
$$
Consequently, we obtain from \eqref{3.Haux}:
$$
  H((u,\Phi)(t)|(\bar{u},\bar\Phi)(t))
  \le C\int_0^t H(u,\Phi|\bar{u},\bar{\Phi})ds,
$$
and Gronwall's lemma finishes the proof.


\section{Remarks on the uniqueness of solutions}\label{sec.rem}

\begin{remark}[Uniqueness of weak solutions]\label{rem.uniq}\rm
The uniqueness of weak solutions for our model is more delicate than for the model of \cite{GeJu18}, even in the case  $D_i=z_i=1$ for $i=1,\ldots,n$. The reason is that we cannot use simple $L^2(\Omega)$ estimations. Instead, we use the $H^{-1}(\Omega)$ method under the (restrictive) condition that $\na\Phi\in L^\infty(\Omega_T)$. 
This regularity holds if the Dirichlet and Neumann boundaries do not intersect and if $\pa\Omega\in C^{1,1}$, $f\in L^p(\Omega)$, and $\Phi^D\in W^{2,p}(\Omega)$ for some $p>3$. Indeed, we conclude from elliptic regularity \cite[Theorem 3.17]{Tro87} that $\Phi\in L^\infty(0,T;W^{2,p}(\Omega))\hookrightarrow L^\infty(0,T;W^{1,\infty}(\Omega))$.
We also assume that $\sum_{i=1}^n r_i(u)=0$.
Summing \eqref{1.mass} over $i=1\ldots,n$, the pair $(u_0,\Phi)$ solves 
\begin{equation}\label{a.1}
  \pa_t u_0 = \diver(\na\log u_0 - (1-u_0)\na\Phi), \quad
  \lambda^2(\ell^2\Delta-1)\Delta\Phi = 1-u_0+f(x)\quad\mbox{in }
  \Omega,
\end{equation}
together with the corresponding initial and boundary conditions \eqref{1.ic}--\eqref{1.bcphi}. We claim that this system has at most one solution. Let $(u_0,\Phi)$ and $(v_0,\Psi)$ be two weak solutions to this problem and let $\chi\in L^2(0,T;H^1(\Omega))$ be the unique solution to $-\Delta\chi=u_0-v_0$ in $\Omega$, $\na\chi\cdot\nu=0$ on $\pa\Omega$. This solution exists since $\int_\Omega(u_0-v_0)dx=0$ because of mass conservation. We use $\chi$ as a test function in the first equation of \eqref{a.1}:
\begin{align*}
  \frac12\frac{d}{dt}&\int_\Omega|\na\chi|^2 dx
  + \int_\Omega(\log u_0-\log v_0)(u_0-v_0)dx \\
  &= \int_\Omega\big(-(u_0-v_0)\na\Phi + (1-v_0)\na(\Phi-\Psi)\big)
  \cdot\na\chi dx.
\end{align*}
Using $(\log u_0-\log v_0)(u_0-v_0)\ge 4(\sqrt{u_0}-\sqrt{v_0})^2$
and $|u_0-v_0|=|\sqrt{u_0}+\sqrt{v_0}||\sqrt{u_0}-\sqrt{v_0}|\le 2|\sqrt{u_0}-\sqrt{v_0}|$, we find that 
\begin{align*}
  \frac12\frac{d}{dt}\int_\Omega|\na\chi|^2 dx
  + 4\int_\Omega(\sqrt{u_0}-\sqrt{v_0})^2 dx
  &\le C\|\sqrt{u_0}-\sqrt{v_0}\|_{L^2(\Omega_T)}
  \|\na\Phi\|_{L^\infty(\Omega_T)}\|\na\chi\|_{L^2(\Omega_T)} \\
  &\phantom{xx}
  + C\|\na(\Phi-\Psi)\|_{L^2(\Omega_T)}\|\na\chi\|_{L^2(\Omega_T)} \\
  &\le 2\|\sqrt{u_0}-\sqrt{v_0}\|_{L^2(\Omega_T)}^2
  + C\|\na\chi\|_{L^2(\Omega_T)}^2,
\end{align*}
where we used the elliptic estimate $\|\na(\Phi-\Psi)\|_{L^2(\Omega_T)}\le C\|u_0-v_0\|_{L^2(\Omega_T)}$ and the assumption $\|\na\Phi\|_{L^\infty(\Omega_T)}\le C$. We conclude from Gronwall's lemma that $\na\chi(t)=0$ and consequently $u_0(t)=v_0(t)$ and $\Phi(t)=\Psi(t)$ for $t>0$. Now, the equation
\begin{equation}\label{a.2}
  \pa_t u_i = \diver(\na u_i - u_i\na(\log u_0-\Phi))
\end{equation}
can be interpreted as a drift-diffusion equation for $u_i$ with given $(u_0,\Phi)$. The regularity $\na\log u_0-\Phi\in L^2(\Omega_T)$ is sufficient for the application of Gajewski's entropy method; see \cite[Sec.~3]{GeJu18}. Thus, there exists at most one solution $u_i$ to \eqref{a.2} with the corresponding initial and boundary conditions.
\qed\end{remark}

\begin{remark}[Weak--strong uniqueness in the presence of reaction terms]\label{rem.reac}\rm 
We claim that Theorem \ref{thm.wsu} holds for reaction rates  $r_i:\overline\dom\to\R$, which are Lipschitz continuous and quasi-positive (i.e.\ $r_i(u)\ge 0$ for all $u\in\dom$ with $u_i=0$) such that the total reaction rate is nonnegative, i.e.\ $\sum_{i=1}^n r_i(u)\le 0$ for all $u\in\dom$, and that $r_i(u)\log u_i=0$ if $u_i=0$. Proceeding as in Step 1 of the proof of Theorem \ref{thm.wsu} and taking into account Remark \ref{rem.feir}, we need to estimate additionally the expression
\begin{align*}
  & R = \int_\Omega\sum_{i=1}^n r_i(u)(w_i-\bar{w}_i)dx =: R_1 + R_2, \quad\mbox{where} \\
  & R_1 = \int_\Omega\sum_{i=1}^n\bigg\{
  r_i(u)\bigg(\log\frac{u_i}{\bar{u}_i}
  - \log\frac{u_0}{\bar{u}_0}\bigg)
  - r_i(\bar{u})\bigg(\frac{u_i}{\bar{u}_i}-\frac{u_0}{\bar{u}_0}\bigg)
  \bigg\}dx, \\
  & R_2 =  \int_\Omega\sum_{i=1}^n
  z_i(r_i(u)-r_i(\bar{u}))(\Phi-\bar\Phi)dx.
\end{align*}
The assumptions on $r_i$ imply that $r_i(u)\log u_i$ is integrable. Therefore, following \cite[p.~202f]{Fis17},
\begin{align*}
  R_1 &= \int_\Omega\sum_{i=1}^n\bigg\{r_i(u)\bigg(
  \log\frac{u_i}{\bar{u}_i} - \frac{u_i}{\bar{u}_i} + 1\bigg)
  - (r_i(u)-r_i(\bar{u}))\bigg(\frac{u_i}{\bar{u}_i}-1\bigg) \\
  &\phantom{xx}
  - r_i(u)\bigg(\log\frac{u_0}{\bar{u}_0} - \frac{u_0}{\bar{u}_0} 
  + 1\bigg)
  + (r_i(\bar{u})-r_i(\bar{u}))\bigg(\frac{u_0}{\bar{u}_0}-1\bigg)
  \bigg\}dx.
\end{align*}
We deduce from $0\ge \log z-z+1\ge -|z-1|^2/\min\{1,z\}$ for $z>0$
that
\begin{align*}
  R_2 &\le \int_\Omega\sum_{i=1}^n\bigg\{C_Ru_i
  \frac{|u_i-\bar{u}_i|^2}{\bar{u}_i\min\{u_i,\bar{u}_i\}}
  + \frac{C}{\bar{u}_i}|r_i(u)-r_i(\bar{u})||u_i-\bar{u}_i| \\
  &\phantom{xx}- r_i(u)\bigg(\log\frac{u_0}{\bar{u}_0} - \frac{u_0}{\bar{u}_0} + 1\bigg)
  + \frac{C}{\bar{u}_0}|r_i(\bar{u})-r_i(\bar{u})||u_0-\bar{u}_0|
  \bigg\}dx \\
  &\le C\int_\Omega\sum_{i=1}^n|u_i-\bar{u}_i|^2 dx
  - \int_\Omega\sum_{i=1}^n r_i(u)\bigg(\log\frac{u_0}{\bar{u}_0} - \frac{u_0}{\bar{u}_0} + 1\bigg)dx \\
  &\le C\int_\Omega\sum_{i=1}^n|u_i-\bar{u}_i|^2 dx.
\end{align*}
where we used in the last step the assumption $\sum_{i=1}^n r_i(u)\le 0$.
Furthermore, by the Lipschitz continuity of $r_i$, the Poincar\'e inequality, and the elliptic estimate for the Poisson--Fermi equation,
$$
  R_2 \le C\sum_{i=1}^n\|u_i-\bar{u}_i\|_{L^2(\Omega)}
  \|\na(\Phi-\bar{\Phi})\|_{L^2(\Omega)}
  \le C\sum_{i=1}^n\|u_i-\bar{u}_i\|_{L^2(\Omega)}^2.
$$
Thus, estimate \eqref{3.Haux} is still valid with another constant,
and Theorem \ref{thm.wsu} follows.
\qed\end{remark}


\end{document}